\RequirePackage{fix-cm}
\documentclass[smallextended]{svjour3}

\smartqed

\usepackage{amsmath,amssymb,graphicx,booktabs,placeins,float}
\usepackage[numbers,sort&compress]{natbib}
\usepackage{hyperref}
\graphicspath{{figures/}}

\makeatletter
\def\fps@figure{!t}
\def\fps@table{!t}
\makeatother
\providecommand{\doi}[1]{\href{https://doi.org/#1}{\nolinkurl{https://doi.org/#1}}}

\newcommand{\E}{\mathbb{E}}
\newcommand{\bx}{\mathbf{x}}
\newcommand{\norm}[1]{\left\lVert#1\right\rVert}

\makeatletter
\newcounter{algorithm}
\renewcommand{\thealgorithm}{\arabic{algorithm}}

\newcommand{\algorithmname}{Algorithm}
\def\fps@algorithm{t}
\def\ftype@algorithm{4}
\def\ext@algorithm{loa}
\def\fnum@algorithm{\algorithmname~\thealgorithm}
\newenvironment{algorithm}{\@float{algorithm}}{\end@float}
\makeatother

\hypersetup{
  bookmarksdepth=3,
  hidelinks,
  pdfborder={0 0 0}
}

\journalname{Journal of Scientific Computing}

\makeatletter
\renewcommand{\makeheadbox}{}
\makeatother

\begin{document}

\title{Residual-Christoffel Sampling for Random Feature Collocation of Linear PDEs}
\titlerunning{Residual-Christoffel Sampling for RFM Collocation}

\author{Jiale Linghu \and Yangshuai Wang}
\authorrunning{J. Linghu and Y. Wang}

\institute{J. Linghu \at
School of Mathematics and Statistics, Xidian University, Xi'an, China
\and
Y. Wang \at
Department of Mathematics, National University of Singapore, Singapore\\
\email{yswang@nus.edu.sg}}

\date{}

\maketitle

\begin{abstract}
Random feature collocation fixes a randomly generated trial space and determines its coefficients from
a linear least-squares system. Stability then depends on whether the sampled residual equations
represent the geometry induced by the differential operator. We construct an operator-aware
discretization in which the operator-applied features determine both the collocation measure and a
coefficient whitening map. The randomized scheme combines a residual-Christoffel density with
inverse-density weights, while a deterministic scalar-row alternative maximizes successive
regularized log-determinant increments. Conditional on the realized trial space, the sampled whitened
interior Gram is a spectral approximation to the reference Gram on the retained residual space, with
sample complexity linear in the retained dimension up to a logarithmic factor. For uniformly analytic
residual kernels, the associated operator has stretched-exponentially decaying eigenvalues and
ridge effective dimension that is polylogarithmic in the inverse ridge scale. Experiments on scalar and vector equations, varied
geometries, and one to three spatial dimensions show that residual-space sampling and whitening
produce numerically full-rank transformed systems with substantially smaller condition numbers and
iteration counts. The deterministic construction attains the lowest errors at the smallest scalar
sample sizes.
Residual-space geometry therefore yields a principled design for stable strong-form random feature
collocation.
\end{abstract}

\section{Introduction}
Random feature method (RFM) collocation converts a linear partial differential equation (PDE) into a
linear least-squares problem once the trial functions have been drawn. The fixed trial space makes the
coefficient solve explicit, but it also exposes a discretization question: which residual equations
represent the action of the differential operator on that space? Geometric coverage of the domain can
miss weakly observed residual directions and produce a poorly conditioned empirical norm. We address
this problem by constructing the collocation measure, equation weights, and coefficient coordinates
from the geometry of the operator-applied trial functions.

Let $\Omega$ be the computational domain, $L$ the interior differential operator, $B$ the boundary
operator, and $f$ and $g$ the corresponding data in
$Lu=f$ in $\Omega$ and $Bu=g$ on $\partial\Omega$. For fixed random trial functions
$\{\varphi_m\}_{m=1}^M$, write
\[
  u_M(\bx)=\sum_{m=1}^M c_m\varphi_m(\bx),\qquad
  \psi(\bx)=\big(L\varphi_1(\bx),\ldots,L\varphi_M(\bx)\big)^\top,
\]
where $M$ is the number of features, $c_m$ are the unknown coefficients, $\bx\in\Omega$, and
${}^\top$ denotes transpose. An interior point $\bx_i$ contributes the row $\psi(\bx_i)^\top$;
boundary rows are obtained by applying $B$ at points on $\partial\Omega$. The Gram matrix generated by
$\psi$ describes the residual directions that the collocation equations must resolve.

RFM collocation has been developed for stationary and time-dependent PDEs
\cite{ChenE2022RFM,ChenLuo2023TimeRFM}. Related fixed-feature solvers include physics-informed extreme
learning machines (ELMs) \cite{DwivediSrinivasan2020PIELM}, local and high-dimensional ELMs
\cite{DongLi2021LocELM,WangDong2024HighDimELM}, and local randomized neural-network discontinuous
Galerkin methods \cite{SunDongWang2024LRNNDG}. Strong-form RFM also shares fixed trial functions and
collocated operator evaluations with Kansa-type and radial-basis-function methods
\cite{Kansa1990MQII,Fasshauer1997RBFPDE,HonSchaback2001}. Across these settings, stable coefficient
recovery requires the discrete residual norm to capture the relevant trial-space directions.

More broadly, operator- and regime-adapted classical solvers show how numerical performance can
improve when the discretization reflects the PDE structure, as in uniformly accurate treatments of
highly oscillatory Dirac equations and nonuniform-FFT methods for nonlocal Schr\"odinger interactions
\cite{BaoCaiJiaTang2016MTI,BaoCaiJiaTang2017Dirac,BaoJiangTangZhang2015NUFFT}. Our construction
brings the same design principle to the selection and scaling of residual equations in fixed
random-feature spaces.

Existing point-selection methods encode different information. Geometric and low-discrepancy rules
cover the physical domain. Residual-adaptive physics-informed neural network (PINN) methods refine
regions where the current approximation violates the equation
\cite{LuDeepXDE2021,Wu2023PINNSampling,GaoYanZhou2023FIPINN}. PDE-greedy kernel methods select operator
functionals through power-function or current-error criteria \cite{WenzelPDEgreedy2025}. Optimal
weighted least squares and Christoffel sampling provide trial-space-aware designs for regression
\cite{CohenMigliorati2017,NarayanJakeman2017,HamptonDoostan2015,Adcock2024}. The Christoffel Sampling
for Machine Learning (CS4ML) framework further permits general linear observations, including
differential observations, and develops Christoffel adaptive sampling for PINNs \cite{CS4ML2023}.
Christoffel adaptive sampling has also been paired with sparse random-feature expansions for
multivariate function approximation, where the sampling density is formed from the feature values and
the features are subsequently sparsified \cite{CASSRFE2026}.

The fixed linear structure of RFM collocation permits a direct operator-residual construction. In the
PINN realization of CS4ML, the computable surrogate used for interior sampling omits the differential
operator \cite{CS4ML2023}. Here every feature $L\varphi_m$ is explicit, so a
reference quadrature determines the operator-residual Gram, its retained eigenspace, and the associated
leverage scores before the coefficient solve. The same Gram defines a residual-Christoffel sampling
density and a coefficient whitening map. This coupling targets the reference residual norm over all
retained directions and yields a sampling measure that can be reused
for right-hand sides sharing the operator, domain, and fixed trial space. Randomized least-squares
preconditioners, including count-sketch preconditioning of the random-feature system itself, instead
act on an assembled row set with a fixed point set
\cite{AvronToledo2010,MengSaundersMahoney2014,TanChen2026Precond},
while feature filtering changes the retained columns and hence the approximation space
\cite{RRQR2025}.

For a reference probability measure $\mu$ on $\Omega$, define
\[
  G=\E_{\mu}[\psi(\bx)\psi(\bx)^\top],\qquad
  K_M^L(\bx)=\psi(\bx)^\top G_r^\dagger\psi(\bx),
\]
where $\E_\mu$ denotes expectation with respect to $\mu$, $G$ is the population residual Gram, $r$ is
the retained residual rank, $G_r^\dagger$ is its rank-$r$ truncated spectral inverse, and $K_M^L$ is the
residual-Christoffel function. The randomized method samples from the density $K_M^L/r$ relative to
$\mu$ and uses
inverse-density weights together with the whitening map derived from $G$. For scalar residual rows, a
deterministic companion selects the candidate giving the largest regularized log-determinant
increment. The analysis establishes retained-space Gram concentration and conditioning. For uniformly
analytic residual kernels, it further connects the effective residual dimension to spectral decay of
the associated random-feature operator.

The main contributions are as follows.
\begin{enumerate}
\item We construct residual-Christoffel sampling, inverse-density weighting, and a matching whitening
map from the fixed operator-residual features, together with deterministic leverage-greedy selection.
\item We prove retained-space Gram concentration and conditioning, and relate the sampling dimension
to the spectrum and ridge effective dimension of the random-feature residual-kernel operator.
\item Numerical experiments across scalar and vector PDEs, varied geometries, and one to three
dimensions demonstrate improved conditioning, faster iterative solution, and the lowest small-sample
errors from greedy selection in the scalar tests.
\end{enumerate}

The remainder of this paper is organized as follows. Section~\ref{sec:form} formulates the
residual-space discretization, Section~\ref{sec:algorithms}
develops the sampling algorithms, and Section~\ref{sec:analysis} gives the conditioning and spectral
analysis. Section~\ref{sec:results} presents the numerical experiments, followed by the conclusions
and outlook in Section~\ref{sec:discussion}.

\section{Problem formulation and residual-space discretization}\label{sec:form}
We distinguish the reference residual Gram, sampled interior Gram, boundary-augmented least-squares
system, and coefficient map because they govern sampling, assembly, and solver coordinates separately.

\subsection{Model problem and fixed-feature trial space}
Let $\Omega\subset\mathbb R^d$ be a bounded computational domain, where $\mathbb R$ denotes the real
numbers and $d$ is the spatial dimension, and let $L$ and $B$ be linear
interior and boundary operators, respectively. We consider
\begin{equation}\label{eq:model-problem}
  Lu=f\quad\text{in }\Omega,\qquad Bu=g\quad\text{on }\partial\Omega,
\end{equation}
for the unknown solution $u$, using scalar notation. Systems of equations are handled by stacking component equations and their
coefficient blocks, so one physical point may contribute several equation rows, as in the elasticity
experiments. Let $\mu$ be a reference probability measure on
$\Omega$ and let $\nu$ be a reference probability measure on $\partial\Omega$. Unless stated
otherwise, these are normalized volume and boundary reference measures. The positive integer $M$ is
the number of realized trial functions, and the analysis is conditional on the fixed realization
$\{\varphi_m\}_{m=1}^M$. We assume that $L\varphi_m\in L^2(\mu)$ and
$B\varphi_m\in L^2(\nu)$, with pointwise representatives on the collocation sets, and that the data
$f$ and $g$ can be evaluated at the corresponding points. Here $L^2(\mu)$ and $L^2(\nu)$ are the
usual spaces of square-integrable functions with respect to the indicated measures.

Throughout, $\|\cdot\|_2$ denotes the Euclidean norm for vectors and the spectral norm for matrices;
${}^\top$ denotes transpose, $I_k$ denotes the $k\times k$ identity matrix, and $\E$ and $\Pr$
denote expectation and probability.

In computations, $\mu$ is represented by a probability-normalized quadrature with $Q$ nodes
$\tilde{\bx}_q$ and nonnegative weights $w_q$,
$\mathcal Q=\{(\tilde{\bx}_q,w_q)\}_{q=1}^Q$, where $\sum_qw_q=1$, so that, for an integrable
function $h$, $\E_\mu h(\bx)$ is approximated by
$\sum_qw_qh(\tilde{\bx}_q)$. When a boundary Gram is required, an independent quadrature with
$Q_b$ boundary nodes $\tilde y_q$ and weights $w_q^b$ represents $\nu$:
$\mathcal Q_b=\{(\tilde y_q,w_q^b)\}_{q=1}^{Q_b}$ with $\sum_qw_q^b=1$. All expectations below are
over spatial variables after the feature realization has been fixed.

The RFM trial space is
\begin{equation}\label{eq:rfm-trial-space}
  V_M=\mathrm{span}\{\varphi_1,\ldots,\varphi_M\},\qquad
  u_M(\bx)=\sum_{m=1}^M c_m\varphi_m(\bx).
\end{equation}
For the fixed trial space \eqref{eq:rfm-trial-space}, define the operator-residual feature vector
\begin{equation}\label{eq:operator-residual-vector}
  \psi(\bx)=\big(L\varphi_1(\bx),\ldots,L\varphi_M(\bx)\big)^\top .
\end{equation}

\subsection{Interior residual Gram and sampling objective}
For symmetric matrices $C$ and $D$, write $C\preceq D$ when $D-C$ is positive semidefinite (PSD).
The interior residual Gram associated with \eqref{eq:operator-residual-vector} is
\begin{equation}\label{eq:interior-residual-gram}
  G:=\E_{\bx\sim\mu}[\psi(\bx)\psi(\bx)^\top]\succeq0.
\end{equation}
The Gram in \eqref{eq:interior-residual-gram} is the default object used for residual-space sampling.
Boundary information forms a separate block, and the augmented Gram below provides full
interior-plus-boundary conditioning. Assume $G\ne0$, and let
$(\lambda_j,v_j)$ denote its eigenpairs, ordered so that the eigenvectors are orthonormal and
$\lambda_1\ge\cdots\ge\lambda_M\ge0$. The zero-Gram case contains no interior residual information
and is excluded.

The analysis allows any fixed retained rank $r$ satisfying $1\le r\le\operatorname{rank}(G)$. This
rank is chosen after the feature realization and reference Gram are fixed but before the collocation rows are
sampled. In the computations, the following gap-aware
numerical rule is used. First set
\begin{equation}\label{eq:rank-selection-rule}
\begin{aligned}
  r_0&=\#\{j:\lambda_j>10^{-10}\lambda_1\},\\
  \mathcal J(r_0)&=\big\{j:\max(1,r_0-4)\le j\le\min(M-1,r_0+4),\ \lambda_j>0\big\},\\
  r&=\begin{cases}
  \displaystyle\min\operatorname*{arg\,max}_{j\in\mathcal J(r_0)}
       \frac{\lambda_j}{\lambda_{j+1}},
       &2\le r_0\le M-1,\\[2mm]
  r_0, &r_0\in\{1,M\}.
  \end{cases}
\end{aligned}
\end{equation}
Here $\#$ denotes the cardinality of a finite set.
When $\lambda_j>0=\lambda_{j+1}$, the ratio in \eqref{eq:rank-selection-rule} is interpreted as
$+\infty$. The minimum fixes a tie convention. The analysis requires a fixed $r$ with
$\lambda_r>0$. For $M=1$, the second branch gives $r=1$.
After the rank is chosen, define the truncated inverse and whitening map by
\begin{equation}\label{eq:truncated-whitening}
  \begin{aligned}
  G_r^\dagger&=\sum_{j=1}^r\lambda_j^{-1}v_jv_j^\top,\\
  T_r&=[v_1/\sqrt{\lambda_1},\dots,v_r/\sqrt{\lambda_r}]
      \in\mathbb R^{M\times r},\\
  \psi^\perp(\bx)&=T_r^\top\psi(\bx)\in\mathbb R^r .
  \end{aligned}
\end{equation}
Then $G_r^\dagger=T_rT_r^\top$ and $T_r^\top GT_r=I_r$, so
$\E_\mu[\psi^\perp\psi^{\perp\top}]=I_r$ in the retained eigencoordinates. The corresponding
residual-Christoffel function is
\begin{equation}\label{eq:residual-christoffel-function}
  K_M^L(\bx)=\psi(\bx)^\top G_r^\dagger\psi(\bx)=\|\psi^\perp(\bx)\|_2^2,
  \qquad \E_\mu[K_M^L]=r .
\end{equation}
The randomized sampling rule uses \eqref{eq:residual-christoffel-function} directly. The deterministic rule
uses the same operator-residual rows but updates their leverage metric as points are selected. Points
where $K_M^L=0$ contain no information in the retained coordinates and receive zero probability in the
randomized rule.

For a system with total coefficient dimension $M_{\rm sys}$ and $s$ equation rows per physical point,
collect the residual rows in $\Psi(\bx)\in\mathbb R^{s\times M_{\rm sys}}$. The same construction applies after replacing
$\psi\psi^\top$ by $\Psi^\top\Psi$ in $G$ and replacing $K_M^L$ by
\[
  K_{M,\mathrm{sys}}^L(\bx)
  =\operatorname{tr}\!\big(\Psi(\bx)G_r^\dagger\Psi(\bx)^\top\big)
  =\|\Psi(\bx)T_r\|_F^2,
\]
where $\|\cdot\|_F$ is the Frobenius norm. In particular,
\[
  \E_\mu[K_{M,\mathrm{sys}}^L]
  =\operatorname{tr}(T_r^\top G T_r)=r.
\]
Sampling a point assigns the same inverse-Christoffel weight to all of its equation rows, and its
empirical Gram contribution is
\[
  \omega_iT_r^\top\Psi(\bx_i)^\top\Psi(\bx_i)T_r.
\]
The resulting point contribution is PSD and its spectral norm is
bounded by its trace, so the concentration argument below retains the same $r/N$ summand bound. This
block formulation is used for the elasticity experiment, where $M_{\rm sys}=2M$ and $s=2$.
The population and finite-candidate formulas below apply unchanged after replacing each scalar outer
product by its block counterpart.

For selected interior points or point blocks $\{(\bx_i,\omega_i)\}_{i=1}^{N_i}$, let $\omega_i>0$ denote the sampling
weight before the common assembly normalization introduced below. The empirical whitened Gram is
\begin{equation}\label{eq:empirical-whitened-gram}
  \widehat G^\perp
  = \sum_{i=1}^{N_i}\omega_i\,
  \psi^\perp(\bx_i)\psi^\perp(\bx_i)^\top .
\end{equation}
The sampling target is therefore $\widehat G^\perp\approx I_r$ on the retained operator-residual
coordinates. Thus row selection constructs a weighted empirical approximation of the reference
residual Gram rather than only a geometric covering of $\Omega$. The immediate objective is algebraic
conditioning on the retained interior block. Solution accuracy is evaluated separately. The concentration
analysis uses the sampling weights $\omega_i$; the normalized weights $\bar\omega_i$ below are reserved
for the assembled PDE objective.

\subsection{Boundary block and assembled least-squares system}
The actual PDE discretization also contains boundary equations. For interior points
$\{\bx_i\}_{i=1}^{N_i}$ with $N_i\ge1$, normalize the sampling weights by
\[
  \bar\omega_i=\frac{\omega_i}{\sum_{s=1}^{N_i}\omega_s},
  \qquad \sum_{i=1}^{N_i}\bar\omega_i=1.
\]
Let $\{y_j\}_{j=1}^{N_b}$, with $N_b\ge1$, be the boundary points and let $\eta_j>0$ satisfy
$\sum_{j=1}^{N_b}\eta_j=1$. The assembled objective is
\begin{equation}\label{eq:weighted-ls-objective}
  \begin{aligned}
  \min_{c\in\mathbb R^M}\quad &
  \sum_{i=1}^{N_i}\bar\omega_i\,\big(Lu_M(\bx_i)-f(\bx_i)\big)^2
  +\beta\sum_{j=1}^{N_b}\eta_j\,\big(Bu_M(y_j)-g(y_j)\big)^2,\\
  & u_M=\sum_{m=1}^M c_m\varphi_m .
  \end{aligned}
\end{equation}
Here $\beta>0$ is the total boundary weight relative to the unit-normalized interior block. Uniform and
greedy rules use constant unnormalized interior weights, whereas the randomized rule uses
inverse-Christoffel weights. The normalization from $\omega_i$ to $\bar\omega_i$ does not change the
condition number of the interior block because it is a uniform scalar rescaling, but it fixes the
interior scale relative to the boundary block.

Define
\begin{equation}\label{eq:assembled-blocks}
  \begin{array}{ll}
  A_i(i,m)=L\varphi_m(\bx_i),\quad A_i\in\mathbb R^{N_i\times M}, &
  A_b(j,m)=B\varphi_m(y_j),\quad A_b\in\mathbb R^{N_b\times M}, \\[2mm]
  \bar\Omega_i=\mathrm{diag}(\bar\omega_1,\ldots,\bar\omega_{N_i}), &
  \Omega_b=\mathrm{diag}(\eta_1,\ldots,\eta_{N_b}).
  \end{array}
\end{equation}
Here $\operatorname{diag}$ forms a diagonal matrix from the listed entries. Define the sampled data
vectors $f_i:=(f(\bx_1),\ldots,f(\bx_{N_i}))^\top$ and
$g_b:=(g(y_1),\ldots,g(y_{N_b}))^\top$. The corresponding base weighted least-squares problem is
\begin{equation}\label{eq:assembled-ls-system}
  \min_{c\in\mathbb R^M}\|\mathcal A c-b\|_2^2,\qquad
  \mathcal A:=\begin{bmatrix}
  \bar\Omega_i^{1/2}A_i\\
  \sqrt{\beta}\,\Omega_b^{1/2}A_b
  \end{bmatrix},\qquad
  b:=\begin{bmatrix}
  \bar\Omega_i^{1/2}f_i\\
  \sqrt{\beta}\,\Omega_b^{1/2}g_b
  \end{bmatrix},
\end{equation}
Here $\mathcal A\in\mathbb R^{(N_i+N_b)\times M}$ and $b\in\mathbb R^{N_i+N_b}$.

Whitening or preconditioning introduces a full-column-rank coefficient map
$P_{\rm coef}\in\mathbb R^{M\times m_{\rm tr}}$, where $m_{\rm tr}$ is the dimension of the transformed coefficient
space, and solves for the transformed coefficient vector $\widetilde c\in\mathbb R^{m_{\rm tr}}$:
\begin{equation}\label{eq:transformed-ls-system}
  \min_{\widetilde c\in\mathbb R^{m_{\rm tr}}}\|\mathcal A P_{\rm coef}\widetilde c-b\|_2^2,
  \qquad c=P_{\rm coef}\widetilde c.
\end{equation}
For retained residual whitening, $P_{\rm coef}=T_r$ and $m_{\rm tr}=r$. Because $T_r$ is rectangular when $r<M$, this
choice restricts the coefficient vector to $\operatorname{range}(T_r)$ as well as whitening that
subspace. An untransformed solve in the full coefficient space based on a singular value decomposition (SVD) uses
$P_{\rm coef}=I_M$, where $I_M$ is the $M\times M$ identity matrix; an invertible full-rank
$P_{\rm coef}$ changes only
the coefficient coordinates. Full-rank ridge and boundary-aware transforms are identified explicitly
where they are used.
Reported condition numbers and LSQR iterations refer to the matrix $\mathcal A P_{\rm coef}$ actually
passed to the solver. Experiments that isolate point selection hold $P_{\rm coef}$ fixed, whereas combined stabilization
results are labeled as sampling plus whitening. Within each collocation-sampling comparison, the boundary points,
$\eta_j$, and $\beta$ are shared.

\begin{remark}[Interior and full-system conditioning]\label{rem:interior-full}
The default sampling rule and the concentration result below use the interior residual Gram $G$. For the
condition number of the full interior-plus-boundary matrix, whitening can instead be based on the
coefficient-space Gram induced by a weighted augmented reference measure. For a whitening
parameter $\vartheta\ge0$, define
\begin{equation}\label{eq:boundary-aware-gram}
  \begin{aligned}
  \chi(y)&=\big(B\varphi_1(y),\ldots,B\varphi_M(y)\big)^\top,\\
  G_b&=\E_{y\sim\nu}[\chi(y)\chi(y)^\top],\\
  G_{\vartheta}^{\rm aug}&=G+\vartheta G_b.
  \end{aligned}
\end{equation}
Both terms in \eqref{eq:boundary-aware-gram} are quadratic forms in the same coefficient vector. If the
augmented reference Gram is chosen to match the population counterpart of
\eqref{eq:weighted-ls-objective} under the same normalized reference measures, then $\vartheta=\beta$.
The ablation also uses the proportional normalized
mixture
\[
  \alpha:=\frac{\vartheta}{1+\vartheta},\qquad
  G_\alpha:=(1-\alpha)G+\alpha G_b
  =\frac{1}{1+\vartheta}G_\vartheta^{\rm aug},
\]
whose global factor does not affect a condition number. A whitening transform constructed from $G$
targets the interior quadratic form only. A transform constructed from $G_\vartheta^{\rm aug}$ targets the
combined reference form, but
conditioning of a finite assembled matrix additionally depends on how accurately the sampled interior
and boundary blocks approximate their reference Grams. The concentration theorem below treats the
default retained interior block.
\end{remark}

Thus the reference Gram and sampling weights determine the interior discretization, whereas $\beta$
and $P_{\rm coef}$ determine boundary scaling and solver coordinates, respectively.

\section{Residual-space sampling strategies}\label{sec:algorithms}
The objective is to approximate the weighted residual Gram on the retained eigencoordinates.
Randomized
residual-Christoffel sampling uses the fixed reference Gram $G_r^\dagger$, whereas deterministic
leverage-greedy selection updates the Gram of the rows already selected. Boundary rows are appended
afterward as a fixed block. Throughout this section, $N=N_i$ denotes the number of selected interior
points or point blocks. A point contributes one residual row in the scalar formulation and may
contribute several rows for a system. The notation $K_M$ abbreviates $K_M^L$ when the operator is clear.

\subsection{Randomized residual-Christoffel sampling}
The randomized rule aims to select a small set of weighted interior residual rows whose empirical
whitened Gram \eqref{eq:empirical-whitened-gram} approximates $I_r$ in the retained eigencoordinates.
The population sampling rule draws interior points from the residual-Christoffel measure
\begin{equation}\label{eq:population-christoffel-design}
  dp=(K_M/r)\,d\mu,\qquad
  \omega_i=\frac{r}{N K_M(\bx_i)} .
\end{equation}
Here $p$ is the probability measure obtained by tilting $\mu$ with the residual-Christoffel density
$K_M/r$ on the support where $K_M>0$; it is normalized because $\E_\mu K_M=r$. Conditional on the
fixed features, Gram, and retained rank, the points $\bx_1,\ldots,\bx_N$ are drawn independently with
replacement from $p$. Each sampled row contributes an unbiased $1/N$ share of the identity in the
retained eigencoordinates:
\begin{equation}\label{eq:population-unbiased-row}
  \E_p\!\left[\omega_i\,\psi^\perp(\bx_i)\psi^\perp(\bx_i)^\top\right]
  =\frac1N I_r .
\end{equation}
Moreover, every weighted whitened row contribution has the same trace,
\begin{equation}\label{eq:population-row-norm}
  \omega_i\|\psi^\perp(\bx_i)\|_2^2=\frac{r}{N}.
\end{equation}
Equation~\eqref{eq:population-unbiased-row} gives the mean interior Gram, while
\eqref{eq:population-row-norm} bounds every PSD rank-one summand by $r/N$ in spectral norm.
Before boundary assembly, they are rescaled to $\bar\omega_i$ as defined in Section~\ref{sec:form};
this common scalar rescaling leaves the interior condition number unchanged.

For the discrete sampler, a reference quadrature is first used to approximate the residual Gram and
construct the whitening map. The leverage scores used for sampling are then evaluated on a finite
candidate set $\mathcal C_P=\{(z_q,\varpi_q)\}_{q=1}^{P}$ with $\sum_q\varpi_q=1$, where
$\varpi_q$ is the probability mass assigned to candidate $z_q$. Unless otherwise stated,
this candidate set is drawn from the same normalized Lebesgue/reference measure $\mu$; for a uniform
finite set, $\varpi_q=1/P$. Let $K_q=K_M(z_q)$. The discrete sampling probability of candidate index
$q$ is
\begin{equation}\label{eq:finite-pool-probability}
  S_P=\sum_s \varpi_s K_s,
  \qquad \pi_q=\frac{\varpi_q K_q}{S_P},
\end{equation}
where we require $S_P>0$. The index $q_i$ is drawn independently from the discrete distribution in
\eqref{eq:finite-pool-probability}, and the sampled row is assigned the sampling weight
\begin{equation}\label{eq:finite-pool-weight}
  \omega(q)=
  \frac{S_P}{N K_q} .
\end{equation}
Candidates with $K_q=0$ have zero sampling probability in \eqref{eq:finite-pool-probability} and are
omitted. Conditional on the candidate set and whitening map, the finite-set estimator is unbiased
for the candidate-measure Gram
\begin{equation}\label{eq:finite-pool-unbiased-gram}
  \E\!\left[
  \sum_{i=1}^N \omega(q_i)\,
  \psi^\perp(z_{q_i})\psi^\perp(z_{q_i})^\top
  \mathrel{\Big|}\mathcal C_P,T_r\right]
  =H_P:=\sum_{q=1}^P \varpi_q\,
  \psi^\perp(z_q)\psi^\perp(z_q)^\top .
\end{equation}
Thus \eqref{eq:finite-pool-unbiased-gram} is the finite-set counterpart of
\eqref{eq:population-unbiased-row}, but its target is $H_P$, not necessarily $I_r$. When the candidate
measure approximates the reference residual Gram, $H_P\approx I_r$ and
$S_P=\operatorname{tr}(H_P)\approx r$, where $\operatorname{tr}$ denotes the matrix trace. Hence
\eqref{eq:finite-pool-weight} approaches the population normalization
\[
  \omega(q)\approx \frac{r}{N K_M(z_q)} .
\]
This finite-set rule is used in the experiments. Sampling is with replacement, so a selected point
may appear more than once and is retained as a repeated weighted equation. Christoffel and greedy
sampling rules use the same finite candidate set when they are compared directly. Uniform and residual-adaptive points
are drawn independently from the reference distribution. The
candidate mismatch $\|H_P-I_r\|_2$ is isolated in Remark~\ref{rem:finite-candidate}. Unlike greedy
selection, randomized sampling allows arbitrary $N$ through sampling with replacement.

For experiments that avoid explicit rank truncation, we use $\gamma>0$ and the ridge score
\begin{equation}\label{eq:ridge-christoffel-score}
  K_{M,\gamma}(\bx)=\psi(\bx)^\top(G+\gamma I_M)^{-1}\psi(\bx).
\end{equation}
For the population measure, the normalizing constant is
\[
  d_{\mathrm{eff}}^G(\gamma)=\mathrm{tr}\big(G(G+\gamma I_M)^{-1}\big).
\]
This quantity is the ridge effective dimension of the finite Gram $G$ at ridge parameter $\gamma$.
For a finite candidate set, it is
\[
  S_{P,\gamma}=\sum_q\varpi_qK_{M,\gamma}(z_q).
\]
Equations
\eqref{eq:finite-pool-probability}--\eqref{eq:finite-pool-weight} then apply with $K_M$ and $S_P$
replaced by $K_{M,\gamma}$ and $S_{P,\gamma}$.

The constructions in this subsection sample only the interior residual rows; boundary rows are
appended after converting $\omega_i$ to the normalized assembly weights $\bar\omega_i$ in
\eqref{eq:weighted-ls-objective}.
Algorithm~\ref{alg:christoffel} summarizes the resulting randomized finite-candidate construction.

\input{algorithms}
\begin{algorithm}
\footnotesize
\algorithmruletitle{Randomized residual-Christoffel sampling}{alg:christoffel}
\begin{algorithmlines}
\item Given feature functions $\{\varphi_m\}_{m=1}^M$, operator $L$, reference quadrature
      $\mathcal Q=\{(\tilde{\bx}_q,w_q)\}_{q=1}^{Q}$, candidate set with probability masses
      $\mathcal C_P=\{(z_q,\varpi_q)\}_{q=1}^{P}$, interior sample size $N$, and optional ridge
      $\gamma\ge0$, with $\sum_qw_q=\sum_q\varpi_q=1$.
\item Evaluate the residual feature row
      $\psi(\bx)=(L\varphi_1(\bx),\ldots,L\varphi_M(\bx))^\top$
      on $\mathcal Q$ and $\mathcal C_P$.
\item Form the reference residual Gram
      \[
        G_Q=\sum_{q=1}^{Q}w_q\psi(\tilde{\bx}_q)\psi(\tilde{\bx}_q)^\top .
      \]
\item If $\gamma=0$, compute the eigendecomposition of $G_Q$ and retain the numerical rank $r$. Let
      $V_r\in\mathbb R^{M\times r}$ contain the retained eigenvectors and
      $\Lambda_r\in\mathbb R^{r\times r}$ the corresponding positive eigenvalues. Set
      $T_r=V_r\Lambda_r^{-1/2}$ and
      $G_{Q,r}^{\dagger}=V_r\Lambda_r^{-1}V_r^\top$. For each candidate point, set
      \[
        K_q=\psi(z_q)^\top G_{Q,r}^{\dagger}\psi(z_q)
            =\|T_r^\top\psi(z_q)\|_2^2 .
      \]
\item If $\gamma>0$, set $T_\gamma=(G_Q+\gamma I_M)^{-1/2}$ and
      \[
        K_q=\psi(z_q)^\top(G_Q+\gamma I_M)^{-1}\psi(z_q)
            =\|T_\gamma^\top\psi(z_q)\|_2^2 .
      \]
\item For a system with residual block $\Psi(z_q)$, form
      $G_Q=\sum_qw_q\Psi(\tilde\bx_q)^\top\Psi(\tilde\bx_q)$ and replace the scalar score by the sum
      of the equation-row scores, equivalently $K_q=\|\Psi(z_q)T_r\|_F^2$ in the truncated
      construction. A selected point assigns the same sampling weight to every row in its block.
\item Using these scores, define the finite-set probabilities and inverse-Christoffel sampling weights
      by \eqref{eq:finite-pool-probability}--\eqref{eq:finite-pool-weight}. Candidates with $K_q=0$
      receive zero probability; require $S_P=\sum_q\varpi_qK_q>0$.
\item Draw indices $q_1,\ldots,q_N$ independently from $\{\pi_q\}_{q=1}^{P}$ and set
      $\bx_i=z_{q_i}$ with sampling weight $\omega_i=\omega(q_i)$. Sampling is with replacement, so
      repeated indices are retained as repeated weighted rows.
\item Return sampled interior points $\{\bx_i\}_{i=1}^{N}$ and sampling weights
      $\{\omega_i\}_{i=1}^{N}$ for the interior block, together with $T_r$ when $\gamma=0$ or
      $T_\gamma$ when $\gamma>0$. Before boundary assembly, set
      $\bar\omega_i=\omega_i/\sum_s\omega_s$ as in \eqref{eq:weighted-ls-objective}.
\end{algorithmlines}
\noindent\rule{\textwidth}{0.35pt}
\end{algorithm}

\subsection{Deterministic leverage-greedy selection}
The deterministic rule uses leverage relative to the rows already selected, rather than the fixed
reference score $K_M$. This variant is useful when the candidate residual rows have already been
evaluated and a nonrandom set is desired. Consider an unweighted candidate set
$\mathcal C_P=\{z_q\}_{q=1}^{P}$, choose $N\le P$, and let $\lambda_{\rm g}>0$. Weighted candidate
measures can be incorporated into the corresponding weighted log-determinant design by replacing
$\psi(z_q)$ with $\sqrt{\varpi_q}\psi(z_q)$. Starting from
$G_0=\lambda_{\rm g} I_M$, define the remaining set
$\mathcal C_P^{(k)}=\mathcal C_P\setminus\{\bx_1,\ldots,\bx_k\}$ and choose
\begin{equation}\label{eq:greedy-selection-rule}
  \bx_{k+1}\in\arg\max_{\bx\in\mathcal C_P^{(k)}}
  \psi(\bx)^\top G_k^{-1}\psi(\bx),\qquad
  G_k=\lambda_{\rm g} I_M+\sum_{j=1}^{k}\psi(\bx_j)\psi(\bx_j)^\top .
\end{equation}
Here $G_k$ is the regularized Gram of the residual rows already selected, and
$\psi(\bx)^\top G_k^{-1}\psi(\bx)$ is the leverage of a candidate row in the metric induced by that
selected set. A large score identifies a direction that is weakly represented by the current rows.
For a system of equations, the implementation uses the sum of the equation-row scores as a trace
surrogate and incorporates a selected point by successive rank-one updates. If the rows at a point
form $A(\bx)\in\mathbb R^{s\times M_{\rm sys}}$, the exact block log-determinant increment is
$\log\det(I_s+A(\bx)G_k^{-1}A(\bx)^\top)$; it is not generally equal to the summed row scores. The parameter
$\lambda_{\rm g}$ ensures $G_k\succ0$ from the first step and controls the scale assigned to directions
not yet represented. The selection rule
\eqref{eq:greedy-selection-rule} has a local log-determinant characterization.

\begin{proposition}[Scalar-row regularized greedy log-determinant step]\label{prop:greedy}
Let $\mathcal C_P$ be a candidate set of scalar residual rows and let
$G_k=\lambda_{\rm g}I_M+\sum_{j=1}^{k}\psi(\bx_j)\psi(\bx_j)^\top$ with $\lambda_{\rm g}>0$. Then
\[
  \bx_{k+1}\in\arg\max_{\bx\in\mathcal C_P^{(k)}}
  \psi(\bx)^\top G_k^{-1}\psi(\bx)
\]
maximizes $\log\det G_{k+1}-\log\det G_k$ over the remaining candidates
$\mathcal C_P^{(k)}$.
\end{proposition}
\begin{proof}
For any candidate $\bx$, $G_{k+1}=G_k+\psi(\bx)\psi(\bx)^\top$. The matrix determinant lemma gives
\begin{equation}\label{eq:greedy-determinant-lemma}
  \det G_{k+1}
  =\det G_k\big(1+\psi(\bx)^\top G_k^{-1}\psi(\bx)\big).
\end{equation}
Since the leverage score is nonnegative and $t\mapsto\log(1+t)$ is increasing for $t\ge0$,
\eqref{eq:greedy-determinant-lemma} proves the one-step statement. This gives a local regularized
log-determinant characterization. The Sherman--Morrison identity gives the
corresponding $O(M^2)$ inverse update.
\end{proof}

The selected greedy rows enter the common least-squares assembly with unit sampling weights before
the same boundary and normalization convention is applied. Its accuracy as a function of sample size is
evaluated numerically. For scalar equations, each accepted row gives the largest available
local increase in the regularized log determinant. For systems, the summed row-leverage score is the
trace surrogate described above. Algorithm~\ref{alg:greedy} summarizes this deterministic
finite-candidate construction, and Figure~\ref{fig:designpoints} visualizes the interior collocation sets
produced by the three row-selection rules.

\begin{algorithm}
\footnotesize
\algorithmruletitle{Deterministic leverage-greedy sampling}{alg:greedy}
\begin{algorithmlines}
\item Given residual rows $\psi(z_q)$ on a candidate set
      $\mathcal C_P=\{z_q\}_{q=1}^{P}$, target sample size $N\le P$, and greedy regularization
      $\lambda_{\rm g}>0$.
\item Initialize
      \[
        G_0=\lambda_{\rm g}I_M,\qquad H_0=G_0^{-1}.
      \]
      At each subsequent step, $H_k$ denotes the maintained inverse $G_k^{-1}$.
\item \textbf{for} $k=0,\ldots,N-1$ \textbf{do}
\item \hspace*{1em}Compute the candidate scores
      \[
        \ell_k(q)=\psi(z_q)^\top H_k\psi(z_q),
        \qquad q\notin\{q_1,\ldots,q_k\} .
      \]
\item \hspace*{1em}Choose the candidate with maximal score as the next collocation point,
      \[
        q_{k+1}\in\arg\max_{q\notin\{q_1,\ldots,q_k\}}\ell_k(q),\qquad
        \bx_{k+1}:=z_{q_{k+1}} .
      \]
\item \hspace*{1em}Update the Gram and inverse. With $a=\psi(\bx_{k+1})$, $v=H_ka$, and
      $d_k=1+a^\top v$,
      \[
        G_{k+1}=G_k+aa^\top,\qquad
        H_{k+1}=H_k-\frac{vv^\top}{d_k}.
      \]
\item \textbf{end for}
\item For a vector-valued operator, use the sum of the equation-row scores as a trace surrogate. After
      selecting a point, apply one rank-one update for each equation row. This summed score is not, in
      general, the exact block log-determinant increment.
\item Return deterministic interior points $\{\bx_i\}_{i=1}^{N}$ with unit sampling weights
      $\omega_i=1$. Thus $\bar\omega_i=1/N$ in \eqref{eq:weighted-ls-objective}.
\end{algorithmlines}
\noindent\rule{\textwidth}{0.35pt}
\end{algorithm}

\begin{figure}[!t]\centering
\includegraphics[width=.98\textwidth]{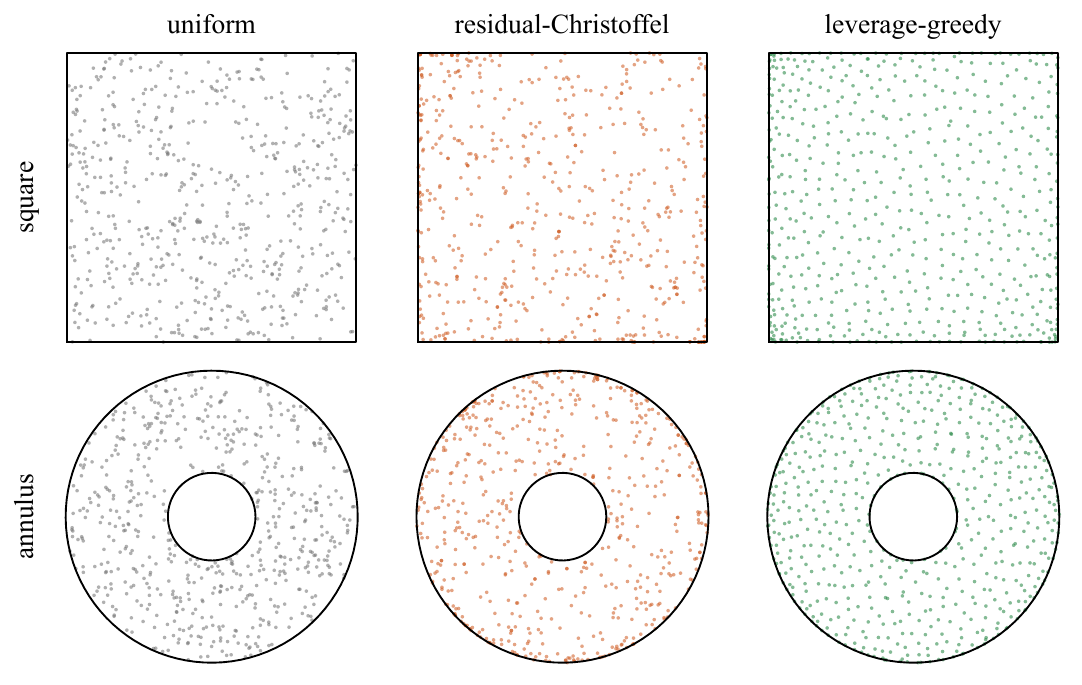}
\caption{Interior collocation sets from the three point-selection rules. Rows show high-frequency Poisson on
the square ($M=500$, $r=252$, $N=600$) and smooth Poisson on the annulus ($M=300$, $r=227$,
$N=600$); columns show uniform, residual-Christoffel, and leverage-greedy sampling. In these examples,
residual-Christoffel sampling visibly reallocates points toward boundary-adjacent high-leverage regions
and may repeat candidates, whereas the without-replacement greedy rule produces a more evenly
separated deterministic set.}
\label{fig:designpoints}
\end{figure}

\subsection{Computational regimes and cost}\label{sec:design-cost}
Let $Q$ be the number of reference quadrature points, $P$ the candidate-pool size, and $r$ the retained
residual dimension. On the reference quadrature, residual-row evaluation costs $O(QM c_L)$, where
$c_L$ is the per-feature cost of evaluating the action of $L$. Forming
$G_Q=\Psi_Q^\top W\Psi_Q$ costs $O(QM^2)$, where $\Psi_Q\in\mathbb R^{Q\times M}$ has row
$q$ equal to $\psi(\tilde{\bx}_q)^\top$ and $W=\operatorname{diag}(w_1,\ldots,w_Q)$,
and diagonalizing this dense $M\times M$ Gram costs $O(M^3)$ with $O(M^2)$ Gram storage.

Candidate-row evaluation costs $O(PM c_L)$. Once the retained eigenspace is available, rank-$r$
Christoffel scores cost $O(PMr)$. Probability normalization costs $O(P)$; an alias table gives
$O(P+N)$ total preprocessing and sampling work. Approximate ridge scores have the same projection cost
when the inverse is represented by a rank-$r$ or sketched factor. An exact full-dimensional ridge
factor costs $O(PM^2)$ to apply over the candidate set.

For the deterministic greedy rule, recomputing the inverse and all scores at every step costs
$O(N(M^3+PM^2))$. With the $P\times M$ candidate-row matrix stored, Sherman--Morrison updates cost
$O(PM+M^2)$ per step, giving $O(N(PM+M^2))$ work and $O(PM+M^2)$ storage after candidate evaluation.
For larger feature counts, a randomized range finder can be applied to the weighted residual matrix
$R_Q=W^{1/2}\Psi_Q$. With sketch size $s_{\rm sk}$, each pass costs $O(QMs_{\rm sk})$; a streamed or matrix-free
implementation uses $O((Q+M)s_{\rm sk})$ auxiliary memory in addition to the data needed to evaluate
residual rows, whereas explicitly storing $R_Q$ requires $O(QM)$ memory. These estimates exclude
boundary assembly and the final least-squares solve.

\section{Conditioning theory}\label{sec:analysis}
Write $K_M:=K_M^L$. Conditional on the fixed features, reference Gram, whitening map, and rank $r$,
inverse-Christoffel sampling makes the whitened interior Gram concentrate about $I_r$. For a finite
candidate set, the target is $H_P$, which equals $I_r$ only when the candidate measure also defines the
whitening Gram. These statements quantify interior conditioning. The
finite-dimensional probability is over the sampled points or point blocks; the operator-limit probability below is over
an independent random-feature realization.

\subsection{Residual-Gram concentration and interior conditioning}

\begin{lemma}[Unbiased residual Gram]\label{lem:unb}
In the scalar formulation, let
\[
  \widehat G^\perp=\sum_{i=1}^N
  \omega_i\psi^\perp(\bx_i)\psi^\perp(\bx_i)^\top ,
\]
and, in the system formulation, replace each summand by
\[
  \omega_iT_r^\top\Psi(\bx_i)^\top\Psi(\bx_i)T_r.
\]
Then
$\E[\widehat G^\perp]=I_r$ in either case.
\end{lemma}
\begin{proof}
The samples are independent and identically distributed from $dp=(K_M/r)\,d\mu$, with
$\omega(\bx)=r/(NK_M(\bx))$. Hence
\[
\begin{aligned}
\E[\widehat G^\perp]
&=N\,\E_{\bx\sim p}\!\big[\omega(\bx)\psi^\perp(\bx)\psi^\perp(\bx)^\top\big]\\
&=N\!\int\!\frac{r}{N K_M}\psi^\perp\psi^{\perp\top}\frac{K_M}{r}\,d\mu\\
&=\int\psi^\perp\psi^{\perp\top}d\mu=I_r,
\end{aligned}
\]
where the last equality follows from $\E_\mu[\psi^\perp\psi^{\perp\top}]=I_r$ on the retained
subspace. The identities are understood on the support $K_M>0$; when $K_M=0$, the rank-one
integrand also vanishes.
For a system, the same cancellation leaves
$T_r^\top\E_\mu[\Psi^\top\Psi]T_r=T_r^\top GT_r=I_r$.
\end{proof}

\begin{remark}[Finite-candidate implementation]\label{rem:finite-candidate}
Let $\mathcal C_P=\{(z_q,\varpi_q)\}_{q=1}^{P}$ be a candidate set and set
\[
  S_P=\sum_{q=1}^P\varpi_q K_M(z_q),\qquad
  \pi_q=\frac{\varpi_qK_M(z_q)}{S_P},\qquad
  \omega(q)=\frac{S_P}{N K_M(z_q)}
\]
on the support where $K_M(z_q)>0$, assuming $S_P>0$. Conditional on $\mathcal C_P$ and the fixed
whitening map $T_r$, the expectation is
\[
  \E[\widehat G^\perp\mid\mathcal C_P]
  =H_P:=\sum_{q=1}^P\varpi_q\psi^\perp(z_q)\psi^\perp(z_q)^\top .
\]
If the same empirical candidate measure is also used for whitening, then $H_P=I_r$ and
$S_P=\mathrm{tr}(H_P)=r$. If whitening is computed from a separate reference quadrature, define the
deterministic mismatch, conditional on the reference quadrature and candidate set, by
\[
  \varepsilon_P^{\rm cand}:=\|H_P-I_r\|_2.
\]
The sampling and quadrature errors then separate as
\begin{equation}\label{eq:candidate-mismatch-bound}
  \|\widehat G^\perp-I_r\|_2
  \le \|\widehat G^\perp-H_P\|_2+\varepsilon_P^{\rm cand}.
\end{equation}
Thus $H_P$ is the conditional candidate target; concentration about $I_r$ additionally requires
control of $\varepsilon_P^{\rm cand}$. Theorem~\ref{thm:sc}, with the stated constant, applies to the population
construction. The same proof applies conditionally to a finite candidate set used for whitening because then
$H_P=I_r$, $S_P=r$, and the same $r/N$ summand bound holds. For a distinct reference quadrature, a bound
about $H_P$ depends additionally on the spectrum of $H_P$ and is not asserted here.
\end{remark}

\begin{theorem}[Population residual-Gram concentration]\label{thm:sc}
Condition on the realized features, population Gram, whitening map $T_r$, and retained rank $r$.
Let the interior points or point blocks be sampled independently according to Lemma~\ref{lem:unb}
and the system extension above. Fix a relative Gram tolerance $\varepsilon\in(0,1)$ and failure
probability $\delta\in(0,1)$. If
\begin{equation}\label{eq:sample-complexity-bound}
  N\ge \frac{3}{\varepsilon^2}\, r\,\ln(2r/\delta),
\end{equation}
then with probability $\ge 1-\delta$,
\begin{equation}\label{eq:gram-concentration-bound}
  \norm{\widehat G^\perp-I_r}_2\le\varepsilon ,
\end{equation}
or, equivalently,
$(1-\varepsilon)I_r\preceq\widehat G^\perp\preceq(1+\varepsilon)I_r$.
\end{theorem}
\begin{proof}
Write $\widehat G^\perp=\sum_{i=1}^N X_i$. In the scalar formulation, the independent summands are
\[
  X_i=\omega_i\psi^\perp(\bx_i)\psi^\perp(\bx_i)^\top\succeq0.
\]
Each term obeys
\[
  \norm{X_i}_2
  =\omega_i\norm{\psi^\perp(\bx_i)}^2
  =\frac{r}{N K_M(\bx_i)}K_M(\bx_i)
  =\frac rN=:L_N,
\]
using $K_M=\norm{\psi^\perp}^2$. This uniform row bound is the role of the inverse-Christoffel
weight. For a point block, the corresponding PSD summand has trace $r/N$ and hence spectral norm at
most $r/N$, so the same argument applies. By Lemma~\ref{lem:unb}, all eigenvalues of
$\E\widehat G^\perp=I_r$ equal one, so
the smallest and largest eigenvalues of this expectation both equal one.
Writing $\lambda_{\min}$ and $\lambda_{\max}$ for the smallest and largest eigenvalues, the matrix
Chernoff inequality \cite{Tropp2012,Tropp2015} gives for $\varepsilon\in(0,1)$
\[
\Pr[\lambda_{\min}(\widehat G^\perp)\le 1-\varepsilon]\le r\,e^{-\varepsilon^2/(2L_N)},\quad
\Pr[\lambda_{\max}(\widehat G^\perp)\ge 1+\varepsilon]\le r\,e^{-\varepsilon^2/(3L_N)}.
\]
Both right-hand sides are at most $\delta/2$ once
$L_N=r/N\le\varepsilon^2/(3\ln(2r/\delta))$. This is exactly
\eqref{eq:sample-complexity-bound}; a union bound yields \eqref{eq:gram-concentration-bound}.
\end{proof}

\begin{corollary}[Conditioning of the retained interior block]\label{cor:cond}
Define $A^\perp\in\mathbb R^{N\times r}$ by
$(A^\perp)_{i,:}=\sqrt{\omega_i}\,\psi^\perp(\bx_i)^\top$. On the event of
Theorem~\ref{thm:sc}, this whitened weighted interior matrix obeys
\begin{equation}\label{eq:retained-block-conditioning}
  \kappa_2(A^\perp)\le\sqrt{\frac{1+\varepsilon}{1-\varepsilon}},
\end{equation}
where $\kappa_2$ denotes the spectral condition number, namely the ratio of the largest to the
smallest singular value. The event in Theorem~\ref{thm:sc} ensures that $A^\perp$ has full column rank.
Once the feature realization and retained rank are fixed, the bound contains
neither $M$ nor the retained
population condition ratio. The latter is $\lambda_1(G)/\lambda_r(G)$. The required sample size still
depends on the possibly $M$-dependent retained rank $r$ through
\eqref{eq:sample-complexity-bound}.
For a system, $A^\perp\in\mathbb R^{Ns\times r}$ is obtained by stacking the $s$ weighted rows from
each sampled point block, and the same bound holds.
\end{corollary}
\begin{proof}
Since $A^{\perp\top}A^\perp=\widehat G^\perp$, the singular values of $A^\perp$ are the square roots of
the eigenvalues of $\widehat G^\perp$. On the event of Theorem~\ref{thm:sc}, these eigenvalues lie in
$[1-\varepsilon,1+\varepsilon]$, which gives \eqref{eq:retained-block-conditioning}.
\end{proof}

\begin{remark}[Interpretation of the conditioning bound]\label{rem:conditioning-scope}
The corollary is a statement about the retained, whitened interior residual block. Its purpose is to
show why whitening removes population anisotropy and residual-Christoffel sampling transfers that
normalization to the sampled block. Before whitening, the retained population ratio
$\lambda_1(G)/\lambda_r(G)$ may be large; the scalar assembly normalization from $\omega_i$ to
$\bar\omega_i$ does not change the interior-block condition number. Unwhitened conditioning,
boundary augmentation, and approximation quality are assessed separately in Section~\ref{sec:results}.
\end{remark}

\subsection{Operator spectrum and ridge effective dimension}\label{sec:effdim}
The retained-rank concentration bound is most informative when the effective residual dimension is small
relative to the ambient feature count. We relate the normalized feature Gram $G/M$ to a
residual-kernel operator and then state the finite-matrix ridge sampling result separately.
Multiplicative convergence of the ridge effective dimension additionally requires trace-resolvent
control beyond this operator-norm approximation.

In this subsection, specialize to single-layer random features
$\varphi(\bx;\theta)=\sigma(w\!\cdot\!\bx+b_\theta)$, where $\sigma$ is the activation function,
$\theta=(w,b_\theta)$ consists of a feature weight $w\in\mathbb R^d$ and bias
$b_\theta\in\mathbb R$, and
$\theta\sim\rho$ is drawn from the parameter law $\rho$. Write
$\zeta_\theta(\bx)=(L\varphi)(\bx;\theta)$, so the residual features are independent and identically
distributed (i.i.d.) random fields. Assume that $\rho$ has bounded support and that the required
derivatives of $\sigma$ are bounded on the corresponding compact affine image of $\Omega$. We also
assume
$R_\zeta:=\sup_{\theta\in\operatorname{supp}(\rho)}\|\zeta_\theta\|_{L^2(\mu)}^2<\infty$; this holds,
for example, for bounded-coefficient differential operators of finite order under the preceding
feature assumptions. For $\bx,\bx'\in\Omega$, define the
residual kernel and its operator on $H=L^2(\mu)$ by
\begin{equation}\label{eq:residual-kernel-operator}
  \begin{aligned}
  k_L(\bx,\bx')&=\E_{\theta}[\zeta_\theta(\bx)\zeta_\theta(\bx')],\\
  (\mathcal T v)(\bx)&=\int_\Omega k_L(\bx,\bx')\,v(\bx')\,d\mu(\bx'),
  \qquad v\in H.
  \end{aligned}
\end{equation}
The operator $\mathcal T$ is self-adjoint, positive semidefinite, and trace-class, with
\[
  \mathrm{tr}(\mathcal T)=\E_\theta\|\zeta_\theta\|_H^2\le R_\zeta.
\]
Let its eigenvalues be
$\tau_1\ge\tau_2\ge\cdots\ge0$. We exclude the trivial case $\mathcal T=0$ and define
$\widetilde d=\mathrm{tr}(\mathcal T)/\tau_1$, the dimensionless trace ratio of $\mathcal T$. We
use $\|\cdot\|_2$ for the induced operator norm on $H$ as well as for the matrix spectral norm. We
reserve $T_r$ for the finite-dimensional whitening
matrix used in the algorithms. For a realized feature set, define
$\Phi_M:\mathbb R^M\to H$ by $\Phi_Mc=\sum_{m=1}^M c_m\zeta_{\theta_m}$. With $*$ denoting the
Hilbert-space adjoint, we have
$G=\Phi_M^*\Phi_M$ and $S_M=\Phi_M\Phi_M^*$, so their nonzero eigenvalues coincide.
This is a sample-covariance-operator setting; sharper concentration results are available under
additional distributional assumptions \cite{KoltchinskiiLounici2017}, while the elementary bound below
is sufficient for the present normalization argument.

\begin{theorem}[Operator limit]\label{thm:oplim}
For $h_1,h_2\in H$, let $h_1\otimes h_2$ denote the rank-one operator
$(h_1\otimes h_2)v=\langle h_2,v\rangle_Hh_1$ for $v\in H$, where
$\langle\cdot,\cdot\rangle_H$ is the inner product of $H$. Let
$S_M=\sum_{m=1}^M\zeta_{\theta_m}\otimes\zeta_{\theta_m}$, and let $\lambda_j(G)$ denote the
eigenvalues of the $M\times M$ feature-index Gram $G$, padded by zeros for $j>M$. For a relative
operator tolerance $\varepsilon_{\rm op}\in(0,1)$ and failure probability $\delta\in(0,1)$, if
\begin{equation}\label{eq:operator-limit-sample-condition}
  M\ge \frac{R_\zeta\,\mathrm{tr}(\mathcal T)}{\delta\varepsilon_{\rm op}^2\tau_1^2}
  =\frac{R_\zeta\widetilde d}{\delta\varepsilon_{\rm op}^2\tau_1},
\end{equation}
then, with probability at least $1-\delta$ over the random-feature realization,
\begin{equation}\label{eq:operator-limit-bound}
  \|M^{-1}S_M-\mathcal T\|_2\le\varepsilon_{\rm op}\tau_1,
  \qquad
  \sup_{j\ge1}|\lambda_j(G)/M-\tau_j|\le\varepsilon_{\rm op}\tau_1.
\end{equation}
In particular, $M^{-1}S_M$ converges to $\mathcal T$ in operator norm almost surely, and the normalized
spectrum of $G$ converges uniformly to that of $\mathcal T$.
\end{theorem}
The proof is given in Appendix~\ref{app:kernel}.

For a constant-coefficient differential operator of order $m_L$,
\[
  L=\sum_{|\alpha|\le m_L}a_\alpha\partial^\alpha,
\]
where $a_\alpha\in\mathbb R$ are the operator coefficients. Here
$\alpha=(\alpha_1,\ldots,\alpha_d)\in\mathbb N_0^d$ is a multi-index,
$\mathbb N_0$ denotes the nonnegative integers,
$|\alpha|=\sum_{j=1}^d\alpha_j$, and
$\partial^\alpha=\partial_{x_1}^{\alpha_1}\cdots\partial_{x_d}^{\alpha_d}$. The residual feature has
the general form
\begin{equation}\label{eq:constant-coefficient-residual-feature}
  \zeta_\theta(\bx)
  =\sum_{|\alpha|\le m_L}a_\alpha w^\alpha
    \sigma^{(|\alpha|)}(w\!\cdot\!\bx+b_\theta).
\end{equation}
Here $w^\alpha=\prod_{j=1}^d w_j^{\alpha_j}$. Formula~\eqref{eq:constant-coefficient-residual-feature}
is for scalar constant-coefficient operators; variable coefficients and vector-valued systems require
the corresponding product and component structure.
Thus operators containing several derivative orders, including Helmholtz and
convection--diffusion operators, require a finite sum rather than a single derivative of $\sigma$.

\begin{theorem}[Spectral decay for uniformly analytic residual kernels]\label{thm:decay}
Assume that $\Omega$ is contained in a Cartesian box $D\subset\mathbb R^d$ and that $\mu$ is a
probability measure supported on $\Omega$. Suppose that, for some tensor-product Bernstein
polyellipse $\mathcal E_\varrho(D)$ with $\varrho>1$, the map
$\xi\mapsto k_L(\xi,\bx')$ is analytic on $\mathcal E_\varrho(D)$ for every $\bx'\in\Omega$ and, for
some $C_k>0$, satisfies
\[
  \sup_{\xi\in\mathcal E_\varrho(D),\,\bx'\in\Omega}|k_L(\xi,\bx')|\le C_k.
\]
Then there are constants $C_1\ge\tau_1$ and $c_1>0$, depending on
$(C_k,\varrho,D,d)$, such that
\begin{equation}\label{eq:spectral-decay-bound}
  \tau_j\le C_1 e^{-c_1 j^{1/d}},\qquad j\ge1.
\end{equation}
\end{theorem}
The proof is given in Appendix~\ref{app:kernel}.

The stretched-exponential form is consistent with classical spectral results for analytic kernels
\cite{LittleReade1984} and related kernel-eigenfunction approximation theory
\cite{SantinSchaback2016}; the theorem above states the precise one-variable analyticity condition used
here.

A sufficient feature-level condition for the hypothesis of Theorem~\ref{thm:decay} is that every
derivative $\sigma^{(q)}$, $0\le q\le m_L$, occurring in
\eqref{eq:constant-coefficient-residual-feature} be analytic and uniformly bounded on the compact
affine image
\[
  \mathcal Z_\varrho
  :=\{w\!\cdot\!\xi+b_\theta:\xi\in\mathcal E_\varrho(D),\
  (w,b_\theta)\in\mathrm{supp}(\rho)\}.
\]
Sine satisfies this condition on every fixed polyellipse. For tanh, the parameter support and
polyellipse must be chosen so that $\mathcal Z_\varrho$ has positive distance from the pole set
$\{i\pi(k+\tfrac12):k\in\mathbb Z\}$, where $i$ is the imaginary unit and $\mathbb Z$ is the set of
integers. The analytic decay conclusion therefore requires this complex-domain separation from the
pole set. For variable-coefficient and elasticity residual kernels, the same conclusion follows
whenever the kernel satisfies the analytic continuation hypothesis of Theorem~\ref{thm:decay};
Section~\ref{sec:results} examines their spectra numerically.

\begin{corollary}[Operator ridge effective dimension]\label{cor:deff}
Under Theorem~\ref{thm:decay},
\begin{equation}\label{eq:operator-effective-dimension-bound}
  d_{\mathrm{eff}}^{\mathcal T}(\ell):=\sum_j\frac{\tau_j}{\tau_j+\ell}
  \le C_2(1+\log(C_1/\ell))^d,
  \qquad 0<\ell\le\tau_1.
\end{equation}
Here $\ell$ is the positive operator ridge parameter, and $C_2$ depends only on $C_1,c_1$, and $d$.
\end{corollary}
The proof is given in Appendix~\ref{app:kernel}.

We next return to the realized finite Gram. The following continuous-row statement is the
finite-dimensional ridge-leverage guarantee used to interpret the sampling experiment.
\begin{theorem}[Ridge leverage concentration]\label{thm:ridge}
Fix $\gamma>0$ and a nonzero finite Gram $G$. Define
\begin{equation}\label{eq:ridge-leverage-score}
  K_{M,\gamma}(\bx)=\psi(\bx)^\top(G+\gamma I_M)^{-1}\psi(\bx),\qquad
  d_{\mathrm{eff}}^G(\gamma)=\mathrm{tr}(G(G+\gamma I_M)^{-1}).
\end{equation}
Sample independently from the probability measure
\begin{equation}\label{eq:ridge-sampling-measure}
  dp_\gamma(\bx)=\frac{K_{M,\gamma}(\bx)}{d_{\mathrm{eff}}^G(\gamma)}\,d\mu(\bx)
\end{equation}
on the support where $K_{M,\gamma}>0$, and use weights
\begin{equation}\label{eq:ridge-sampling-weight}
  \omega_i=\frac{d_{\mathrm{eff}}^G(\gamma)}{N K_{M,\gamma}(\bx_i)}.
\end{equation}
For a relative Gram tolerance $\varepsilon\in(0,1)$ and failure probability $\delta\in(0,1)$, if
\begin{equation}\label{eq:ridge-sample-complexity}
  N\ge \frac{8}{3}\varepsilon^{-2}d_{\mathrm{eff}}^G(\gamma)
  \log(2M/\delta),
\end{equation}
then, with probability at least $1-\delta$, the weighted Gram
$\widehat G=\sum_i\omega_i\psi_i\psi_i^\top$, where $\psi_i=\psi(\bx_i)$, obeys
\begin{equation}\label{eq:ridge-gram-concentration}
  (1-\varepsilon)(G+\gamma I_M)
  \preceq\widehat G+\gamma I_M
  \preceq(1+\varepsilon)(G+\gamma I_M).
\end{equation}
\end{theorem}
\begin{proof}
Let $C_\gamma=(G+\gamma I_M)^{-1/2}G(G+\gamma I_M)^{-1/2}$ and set
\[
  Y_i=(G+\gamma I_M)^{-1/2}
  \omega_i\psi_i\psi_i^\top(G+\gamma I_M)^{-1/2}.
\]
Set $\mathcal R_\gamma=d_{\mathrm{eff}}^G(\gamma)/N$. Then
$\E Y_i=C_\gamma/N$, $\|Y_i\|_2=\mathcal R_\gamma$, $Y_i^2=\mathcal R_\gamma Y_i$, and
$\mathrm{tr}(C_\gamma)=d_{\mathrm{eff}}^G(\gamma)$. Moreover,
$0\preceq Y_i,\E Y_i\preceq \mathcal R_\gamma I_M$, so for $Z_i=Y_i-\E Y_i$,
$\|Z_i\|_2\le \mathcal R_\gamma$ and
\[
  \left\|\sum_i\E Z_i^2\right\|_2
  \le \left\|\sum_i\E Y_i^2\right\|_2
  =\mathcal R_\gamma\|C_\gamma\|_2\le \mathcal R_\gamma.
\]
The self-adjoint matrix Bernstein inequality \cite{Tropp2012,Tropp2015}, together with
$\varepsilon<1$, gives
\[
  \Pr\!\left[\left\|\sum_iY_i-C_\gamma\right\|_2>\varepsilon\right]
  \le 2M\exp\!\left(-\frac{3N\varepsilon^2}
  {8d_{\mathrm{eff}}^G(\gamma)}\right).
\]
Condition \eqref{eq:ridge-sample-complexity} makes this probability at most $\delta$. Undoing the
normalization yields
\[
  \|(G+\gamma I_M)^{-1/2}(\widehat G-G)(G+\gamma I_M)^{-1/2}\|_2\le\varepsilon,
\]
which is equivalent to
$-\varepsilon(G+\gamma I_M)\preceq\widehat G-G\preceq\varepsilon(G+\gamma I_M)$.
Adding $G+\gamma I_M$ gives \eqref{eq:ridge-gram-concentration}. This is a continuous-row result. A
finite candidate construction has an analogous deterministic discrepancy between its empirical ridge
target and the population target, in addition to row-sampling fluctuation. The result is the
counterpart of finite ridge-leverage constructions used in kernel approximation
\cite{AlaouiMahoney2015,CohenMuscoMusco2017}.
\end{proof}

\begin{remark}[Finite versus operator dimensions]\label{rem:operator-finite-deff}
Corollary~\ref{cor:deff} is an operator-level statement, whereas Theorem~\ref{thm:ridge} is conditional
on the realized finite matrix and depends on $d_{\mathrm{eff}}^G(\gamma)$. A multiplicative comparison
between $d_{\mathrm{eff}}^G(M\ell)$ and $d_{\mathrm{eff}}^{\mathcal T}(\ell)$ requires
trace-resolvent concentration in addition to the additive spectral bound
\eqref{eq:operator-limit-bound}.
\end{remark}

For the relative finite-matrix ridges used in the experiments, let the dimensionless parameter
$\rho_{\mathrm{rel}}>0$ specify $\gamma=\rho_{\mathrm{rel}}\lambda_1(G)$, and set
$\ell=\gamma/M$. On the event
\eqref{eq:operator-limit-bound}, if $\varepsilon_{\rm op}<1$, then
\[
  \rho_{\mathrm{rel}}(1-\varepsilon_{\rm op})\tau_1\le\ell
  \le\rho_{\mathrm{rel}}(1+\varepsilon_{\rm op})\tau_1.
\]
The numerical section separately checks normalized-spectrum collapse, the empirical growth of
$d_{\mathrm{eff}}^G(\gamma)$, ridge-relative Gram concentration, and the retained-rank ratio. Over the
reported range $M=400,\ldots,6400$, the observed ratio $r/M$ decreases from $0.44$ to $0.05$.

\section{Numerical results}\label{sec:results}
The numerical study examines effective residual dimension and conditioning. It then considers LSQR
convergence and accuracy versus sample size. Controlled comparisons distinguish sampling from boundary
treatment, preconditioning, feature pruning, trial-space enrichment, and direct solution.
Section~\ref{sec:setup} gives the common setup.

\subsection{Experimental setup}\label{sec:setup}
The compared interior sampling strategies are uniform sampling, randomized residual-Christoffel sampling
(Algorithm~\ref{alg:christoffel}), and deterministic leverage-greedy sampling
(Algorithm~\ref{alg:greedy}). A residual-adaptive control is included in the
accuracy-versus-sample-size experiments:
three enrichment stages sample from a density proportional to the squared current PDE residual. For
elasticity, the score is the squared Euclidean norm of the two equation residuals. All methods use the
same trial space and boundary rows within a test. Ranks, whitening maps, and Christoffel weights follow
the constructions
in Sections~\ref{sec:form} and~\ref{sec:algorithms}, with the default rank chosen by
\eqref{eq:rank-selection-rule}. Throughout the numerical section, $r$ denotes this retained residual
rank, selected separately for each benchmark before the collocation-sampling rule is changed; ratios such as
$N/r$ therefore report selected interior points per retained residual direction. In scalar tests each
point contributes one equation row, whereas an elasticity point contributes a two-row block.
Boundary rows are fixed across samplers.

Unless otherwise stated, boundary points use equal weights normalized
to total boundary weight $\beta=1$, so $\eta_j=1/N_b$ in scalar tests; for vector-valued systems, the
same point weight is applied to each boundary component row. Boundary weighting and boundary-aware
whitening are varied only in the ablation study.

Unless stated otherwise, the trial functions are
$\varphi_m(\bx)=\tanh(w_m\!\cdot\!\bx+b_m)$ with independent
$w_m\sim\operatorname{Unif}([-a_{\rm rf},a_{\rm rf}]^d)$ and
$b_m\sim\operatorname{Unif}([-a_{\rm rf},a_{\rm rf}])$, where $\operatorname{Unif}$ denotes the
uniform distribution. The feature bandwidth is
$a_{\rm rf}=6$ in 1D and in the variable-coefficient, convection--diffusion, elasticity, and geometry tests,
$a_{\rm rf}=5$ in the smooth 2D/3D Poisson tests, and $a_{\rm rf}=8$ in the high-frequency Poisson and $k=6$
Helmholtz accuracy tests. The dense scaling and sketched-construction tests instead use
$\varphi_m(\bx)=\sin(w_m\!\cdot\!\bx+b_m)$ with $a_{\rm rf}=18$ and
$b_m\sim\operatorname{Unif}[-\pi,\pi]$; the $k=16$ Helmholtz test uses the same sine law with
$a_{\rm rf}=28$.

Errors are reported as relative discrete $L^2$ errors on held-out interior test sets. Except in solver
comparisons, accuracy is computed by a singular value decomposition (SVD) in the full coefficient space with
relative cutoff $10^{-13}$; runs of the iterative least-squares solver LSQR
use absolute and relative tolerances $10^{-10}$ with the stated iteration caps. Randomized results are
reported as medians over the repetitions stated in each caption. Spectrum experiments vary both the
feature and reference-quadrature realizations. Accuracy-versus-sample-size experiments fix the feature map, reference
quadrature, boundary set, and test set, and vary the point-selection rule and finite candidate set. The
solver-scaling experiment varies the complete feature, quadrature, and sampling realization. Dotted
large-sample uniform reference errors provide comparison levels. Wall-clock timings were measured on a
server with an Intel Xeon Platinum 8476C central processing unit (CPU) using four numerical-library threads.

All scalar tests impose Dirichlet data from a manufactured solution, with the right-hand side obtained
by applying the stated operator. The Poisson operator is $L=-\Delta$. On unit boxes, the Poisson
solutions are products of sine functions; the high-frequency Poisson accuracy and scaling tests use
$u=\sin(4\pi x)\sin(4\pi y)$ on $[0,1]^2$.
Writing $x$ and $y$ for the Cartesian coordinates, the two Helmholtz tests on $[0,1]^2$ use
wavenumber $k$ in $L=-\Delta-k^2I$, where $I$ is the identity operator. They use $k=6$ with
$u=\sin(2\pi x)\sin(2\pi y)$ and $k=16$ with
$u=\sin(5\pi x)\sin(5\pi y)$. The remaining square-domain tests use
$-\nabla\cdot(a\nabla u)$ with diffusion coefficient
$a=1+0.5\sin(\pi x)\sin(\pi y)$,
$-0.1\Delta u+v\cdot\nabla u$ with convection velocity $v=(2,1)$, and
$u=\sin(2\pi x)\sin(2\pi y)$ in both scalar cases.

The Navier--Cauchy operator is
\[
  \mathcal L_{\rm el}\mathbf u
  =-\mu_{\mathrm{el}}\Delta\mathbf u
   -(\lambda_{\mathrm{el}}+\mu_{\mathrm{el}})\nabla(\nabla\!\cdot\!\mathbf u),
\]
with Lam\'e parameters $\mu_{\mathrm{el}}=\lambda_{\mathrm{el}}=1$ and manufactured displacement
$\mathbf u=(\sin(\pi x)\sin(\pi y),\sin(2\pi x)\sin(2\pi y))^\top$. The subscripts distinguish these elasticity constants
from the reference measure $\mu$ and the Gram eigenvalues $\lambda_j$; $\nabla$ and $\Delta$ denote the
spatial gradient and Laplacian.

Geometry transfer is tested on the L-shaped domain
$[-1,1]^2\setminus((0,1]\times[-1,0))$ and the annulus $0.3\le\|\bx\|_2\le1$.

\subsection{Effective residual dimension and conditioning}
We first examine the dimension of the resolved residual space and then test how its geometry controls
the conditioning and iterative solution of the collocation system.

Figure~\ref{fig:effdim} compares normalized residual-Gram spectra and ridge effective dimensions for
smooth Poisson problems. The leading 2D spectra nearly coincide after normalization, in agreement with
the operator limit in Theorem~\ref{thm:oplim}, while the resolved tails separate as $M$ increases. The
effective dimension remains below the ambient feature count in every dimension, but occupies a
larger fraction of the feature space as the spatial dimension increases.

\begin{figure}[!t]\centering
\includegraphics[width=.48\textwidth]{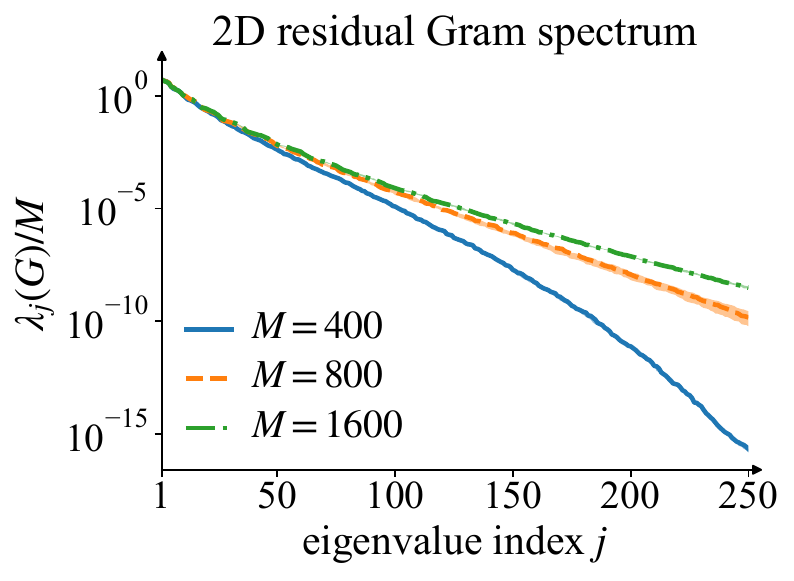}
\includegraphics[width=.48\textwidth]{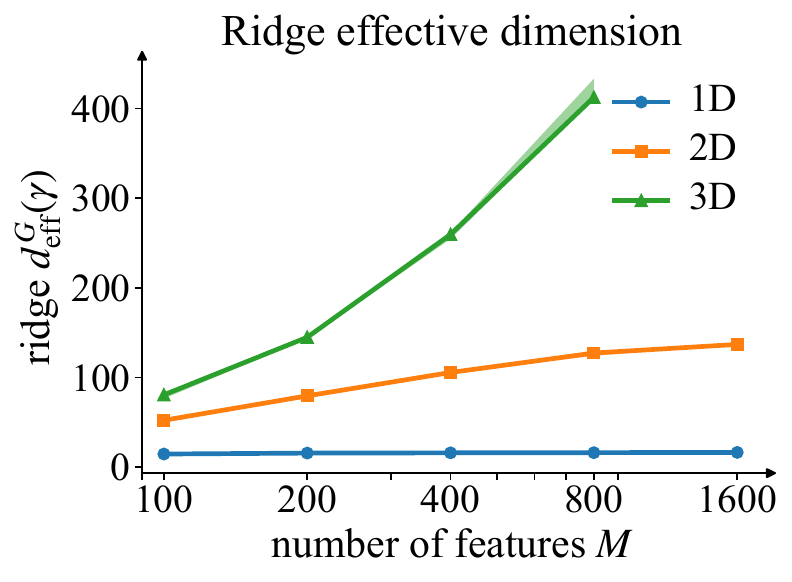}
\caption{Residual spectra and effective dimensions for smooth Poisson tests. Curves show medians over
five independently paired feature/reference draws, with interquartile bands. Left: the leading
normalized 2D residual-Gram spectra are similar across $M$, while the resolved tails separate. Right:
at $\gamma=10^{-6}\lambda_1(G)$, the ridge effective dimension grows slowly in 1D and more rapidly in
2D and 3D; the 3D curve ends at $M=800$.}
\label{fig:effdim}
\end{figure}

Figure~\ref{fig:ridge-conv} tests whether effective dimension organizes the sample size required for
Gram approximation. For representative 1D, 2D, and 3D Poisson residual Grams, it plots
$\|(G+\gamma I_M)^{-1/2}(\widehat G-G)(G+\gamma I_M)^{-1/2}\|_2$ against the normalized sample size
$N/d_{\mathrm{eff}}^G(\gamma)$. The error decreases with normalized sample size in all three cases.
The 1D curve crosses unit relative deviation, the 2D curve approaches it, and the 3D curve remains
larger, revealing dimension-dependent finite-sample constants. Because the candidate pool grows with
$N$, the curves combine row-sampling concentration with increasingly accurate finite-pool resolution.

\begin{figure}[!t]\centering
\includegraphics[width=.88\textwidth]{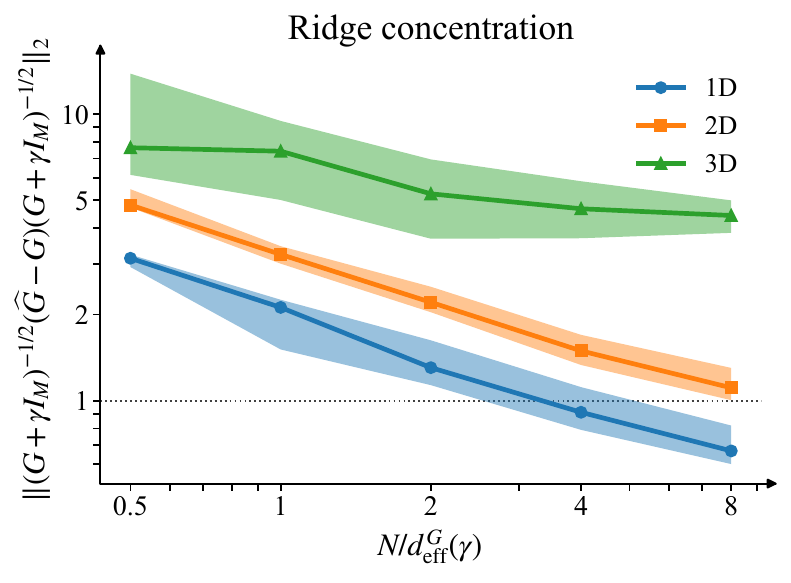}
\caption{Ridge-relative Gram error versus $N/d_{\mathrm{eff}}^G(\gamma)$. Curves show medians over
twenty independent sampling trials, with interquartile bands, for fixed representative Poisson residual Grams
($M=400$ in 1D and $M=600$ in 2D/3D, $\gamma=10^{-6}\lambda_1(G)$), using the unrebalanced
continuous-theory weights and a pool of size $P=\max(30N,6000)$. The dotted line marks unit
relative deviation. Only the 1D median falls below this line within the displayed range.}
\label{fig:ridge-conv}
\end{figure}

Table~\ref{tab:conditioning} reports the effect of sampling and whitening on condition numbers.
Unless stated otherwise, $\kappa_2$ refers to the weighted matrix actually solved, including the common
boundary block. The untransformed uniform matrices are numerically rank deficient in all but the 3D
case, whereas the sampled and whitened matrices are numerically full rank with substantially smaller
condition estimates. The scalar Poisson rows use greedy selection and the remaining rows use
randomized residual-Christoffel sampling; later controls separate point selection from whitening.

\begin{table}[!t]\centering\scriptsize
\caption{Conditioning reduction across benchmarks. $N$ and $N_b$ are interior and boundary point
counts; for elasticity, $M$ is the feature count per displacement component, so the system has $2M$
coefficient columns, and each point contributes one row per equation component. Boundary rows and
their total weight are fixed across sampling rules. Entries are median estimated condition numbers over three
point-selection trials for the boundary-augmented weighted matrix. Untransformed values above the double-precision
resolution scale indicate numerical rank deficiency.}
\label{tab:conditioning}
\setlength{\tabcolsep}{2pt}
\begin{tabular*}{\textwidth}{@{\extracolsep{\fill}}p{0.24\textwidth}lllllll@{}}
\toprule
Problem & size & $r$ & $N$ & $N_b$ & uniform $\kappa_2$ & rule & whitened $\kappa_2$ \\
\midrule
1D Poisson & $M=640$ & 24 & 400 & 2 & $4.3\times10^{17}$ & greedy & $2.1\times10^{3}$ \\
2D Poisson & $M=640$ & 216 & 800 & 300 & $6.4\times10^{17}$ & greedy & $1.9\times10^{3}$ \\
3D Poisson & $M=800$ & 586 & 2000 & 800 & $4.3\times10^{11}$ & greedy & $2.0\times10^{4}$ \\
Helmholtz $k=16$ & $M=4000$ & 270 & 3000 & 800 & $7.5\times10^{19}$ & Christoffel & $3.1\times10^{1}$ \\
Variable-coefficient elliptic & $M=700$ & 264 & 2000 & 400 & $4.3\times10^{17}$ & Christoffel & $5.3\times10^{3}$ \\
Convection--diffusion & $M=700$ & 229 & 2000 & 400 & $3.6\times10^{17}$ & Christoffel & $8.8\times10^{3}$ \\
Navier--Cauchy elasticity & $2M=1000$ & 428 & 1500 & 400 & $6.4\times10^{17}$ & Christoffel & $2.1\times10^{3}$ \\
\bottomrule
\end{tabular*}
\end{table}

Figure~\ref{fig:cond} connects conditioning to iterative convergence on the 2D Poisson benchmark. As
$M$ increases, the untransformed uniform matrix becomes numerically rank deficient and LSQR reaches
its iteration cap. Both residual-space constructions maintain numerical full rank and reduce the solve
to tens or hundreds of iterations. The paired curves show that the
conditioning improvement is directly reflected in LSQR convergence.

\begin{figure}[!t]\centering
\includegraphics[width=.48\textwidth]{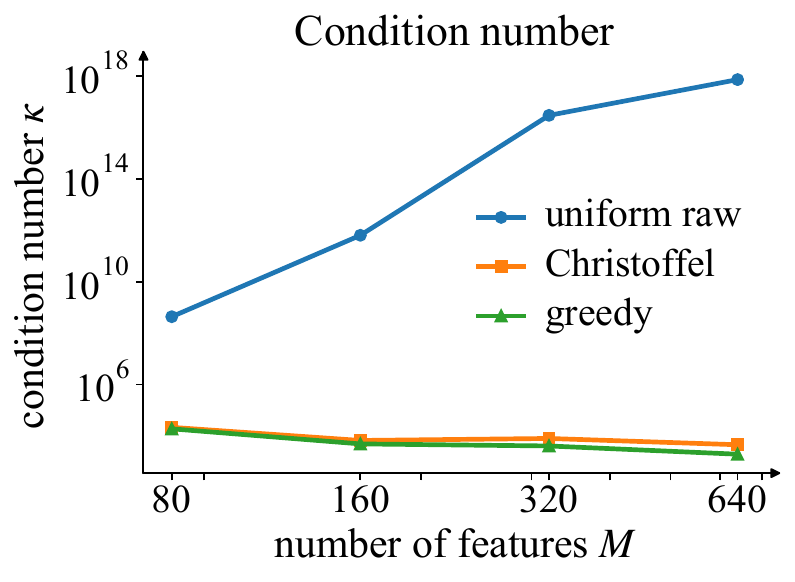}
\includegraphics[width=.48\textwidth]{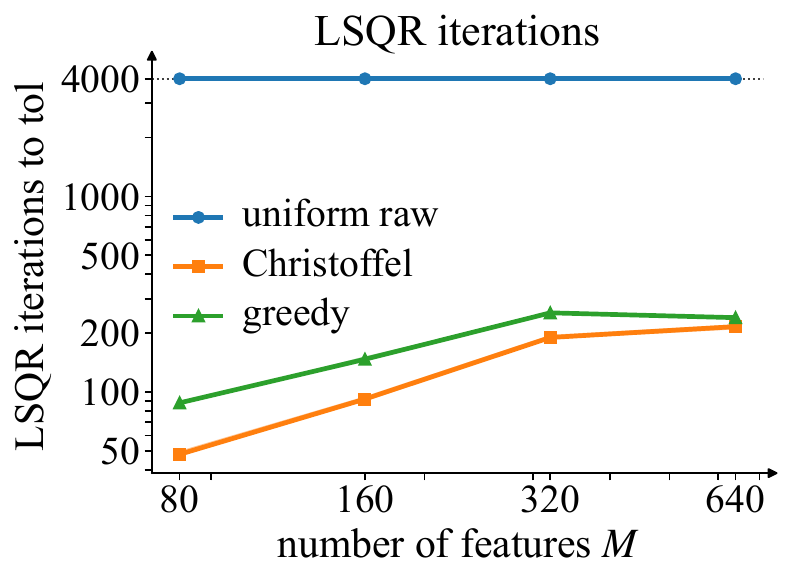}
\caption{Conditioning and LSQR convergence on the 2D Poisson benchmark. Curves show medians over
eight point-selection trials, with interquartile bands. Left: the untransformed uniform matrix becomes
numerically rank deficient, whereas both sampled and whitened matrices remain in the $10^3$--$10^4$ range.
Right: unpreconditioned LSQR reaches the 4000-iteration cap, while the transformed systems meet the tolerance in
48--254 median iterations.}
\label{fig:cond}
\end{figure}

\subsection{Accuracy versus number of collocation points}\label{sec:accuracy-efficiency}
We next examine PDE accuracy for a fixed number of interior collocation points. Conditioning measures
the stability of the residual system, whereas discretization error also depends on the approximation
space; the experiment therefore evaluates both quantities.

Figure~\ref{fig:fixed-ratio-accuracy} compares the three sampling rules at
$N=\lfloor2.5r\rfloor$, where $\lfloor\cdot\rfloor$ denotes the floor function. Uniform sampling is
most accurate for the two smallest feature spaces. Residual-Christoffel and greedy sampling become
more accurate from $M=400$ onward, and the greedy curve reaches the numerical accuracy floor at the
largest feature count.

\begin{figure}[!t]\centering
\includegraphics[width=.88\textwidth]{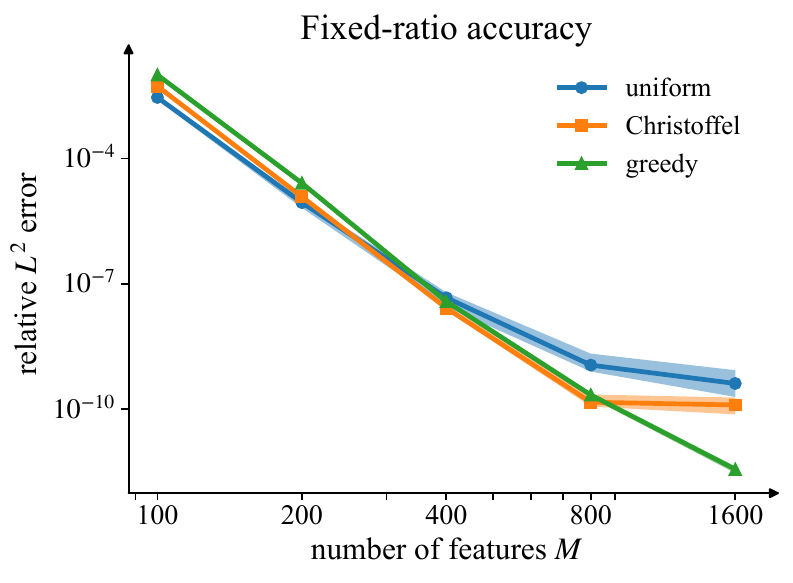}
\caption{Accuracy at fixed sample-to-rank ratio for 2D Poisson with $N=\lfloor2.5r\rfloor$. Curves show medians over 30 point-selection
trials and interquartile bands; the $M=1600$ Christoffel point uses 100 trials. Uniform sampling is
best at the two smallest feature counts, whereas the residual-space sampling rules are more accurate once
$M\ge400$.}
\label{fig:fixed-ratio-accuracy}
\end{figure}

Figure~\ref{fig:multiseed} examines accuracy as a function of $N/r$ for fixed feature maps. At the
smallest ratios, greedy selection has the lowest median error in all four scalar problems and improves
the high-frequency Poisson error by more than an order of magnitude relative to both uniform and
residual-adaptive sampling. The elasticity errors remain comparable across the four methods.
Residual-Christoffel sampling supplies the randomized rule linked to Gram concentration, while the
greedy rule is especially effective when only a small number of scalar equations can be assembled.

\begin{figure}[!t]\centering
\includegraphics[width=.98\textwidth]{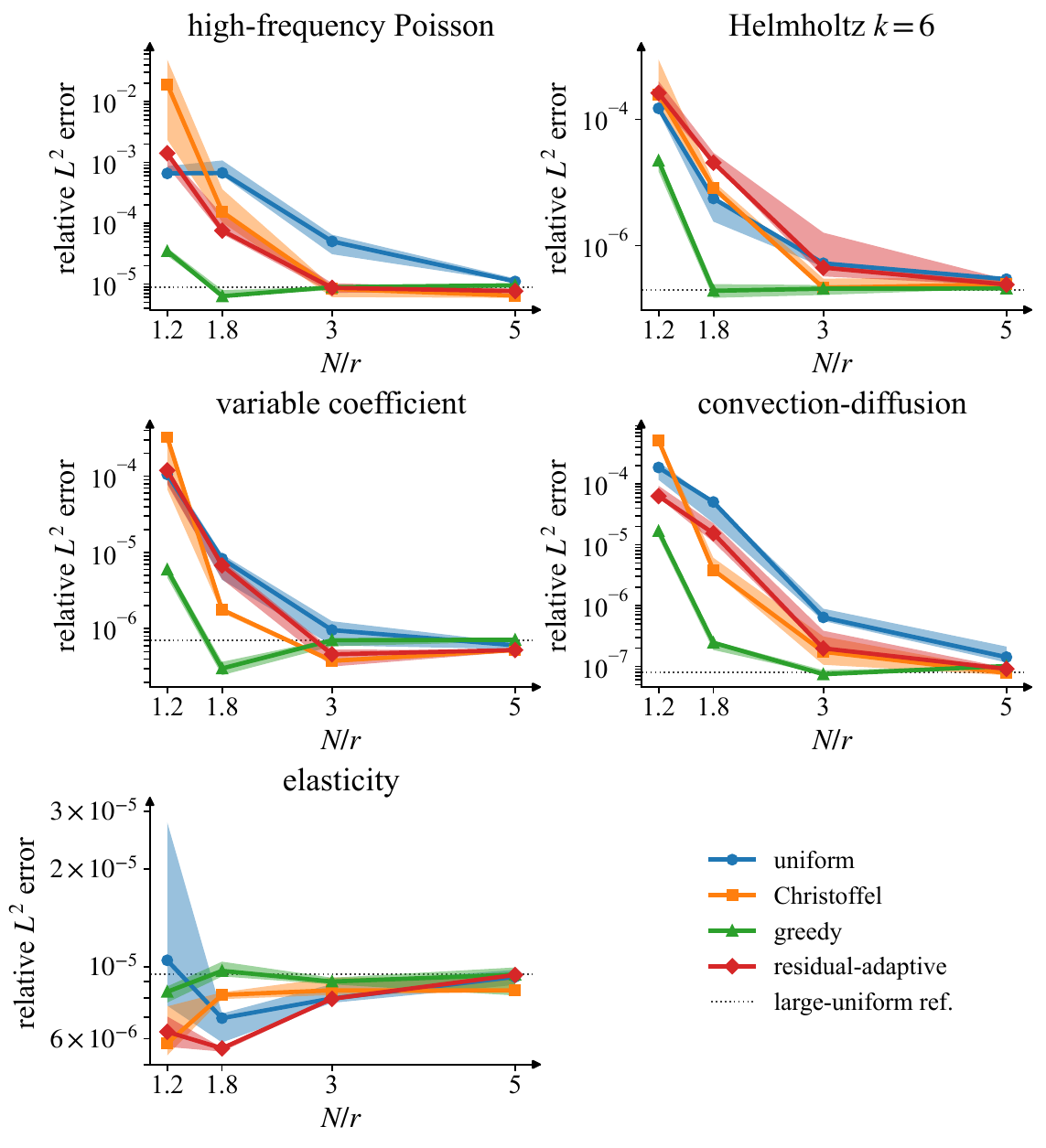}
\caption{Accuracy at prescribed sample-to-rank ratios. Curves show median relative $L^2$ error over
five point-selection and finite-candidate realizations, with interquartile bands. Dotted large-sample uniform
reference levels show the attainable uniform-sampling accuracy. Greedy has the lowest median
error at the smallest sample size in four scalar panels, while the elasticity errors are comparable.}
\label{fig:multiseed}
\end{figure}

The high-frequency Helmholtz results also separate stability from approximation quality. Greedy
selection approaches the large-sample accuracy reference with only a few equations per retained
direction, while residual-Christoffel sampling with whitening produces the well-conditioned
transformed system reported in Table~\ref{tab:conditioning}. The two constructions therefore act on
the same residual geometry but emphasize small-sample coverage and randomized norm preservation,
respectively.

\subsection{Geometry transfer and field-level validation}
Table~\ref{tab:geometry-transfer} examines smooth manufactured solutions on L-shaped and annular
domains for four scalar operators. Every untransformed uniform system reaches the LSQR cap, whereas
residual-Christoffel sampling with whitening produces convergent iterative solves and reduces the
condition estimates by several orders of magnitude. The result shows that the algebraic stabilization
persists across nonrectangular geometries and operator classes.

\begin{table}[!htbp]\centering
\caption{Geometry-transfer experiments for smooth manufactured solutions. Entries are medians over
three point-selection realizations with $M=400$, $N=1400$, and $N_b=350$. Unpreconditioned LSQR on
the uniform system reaches the 2000-iteration cap in every realization; residual-Christoffel sampling
with whitening converges in 114--427 median iterations. The reported error is obtained from an SVD
solve in the full coefficient space on the Christoffel rows.}
\label{tab:geometry-transfer}
\scriptsize
\setlength{\tabcolsep}{2.4pt}
\resizebox{\textwidth}{!}{%
\begin{tabular}{@{}llrrrrrr@{}}
\toprule
Geometry & Operator & $M$ & $r$ & uniform $\kappa_2$ &
Christoffel $\kappa_2$ & LSQR uniform/Christoffel & Rel. $L^2$ \\
\midrule
L-shape & Poisson & 400 & 337 & $6.7{\times}10^{10}$ & $6.8{\times}10^3$ &
2000/390 & $6.1{\times}10^{-4}$ \\
L-shape & Helmholtz $k=6$ & 400 & 300 & $7.7{\times}10^{11}$ &
$4.2{\times}10^1$ & 2000/124 & $2.2{\times}10^{-2}$ \\
L-shape & Variable coeff. & 400 & 341 & $8.2{\times}10^{10}$ &
$1.6{\times}10^4$ & 2000/409 & $1.2{\times}10^{-2}$ \\
L-shape & Convection--diffusion & 400 & 313 & $2.9{\times}10^{11}$ &
$1.4{\times}10^4$ & 2000/427 & $1.1{\times}10^{-3}$ \\
Annulus & Poisson & 400 & 323 & $1.2{\times}10^{11}$ & $6.5{\times}10^3$ &
2000/353 & $9.2{\times}10^{-5}$ \\
Annulus & Helmholtz $k=6$ & 400 & 289 & $1.6{\times}10^{12}$ &
$3.9{\times}10^1$ & 2000/114 & $8.4{\times}10^{-3}$ \\
Annulus & Variable coeff. & 400 & 319 & $1.6{\times}10^{11}$ &
$5.0{\times}10^3$ & 2000/292 & $6.1{\times}10^{-3}$ \\
Annulus & Convection--diffusion & 400 & 303 & $6.3{\times}10^{11}$ &
$1.7{\times}10^4$ & 2000/381 & $5.4{\times}10^{-4}$ \\
\bottomrule
\end{tabular}}
\end{table}

Figure~\ref{fig:solutions} complements the table with one field-level Poisson realization on a square,
an L-shaped domain, and an annulus. All three use residual-Christoffel rows and an SVD solve in the
full coefficient space.
The pointwise maps place the largest errors near the outer boundary of the square, the
reentrant geometry and exterior edges of the L-shape, and the two annular boundary components.

\begin{figure}[!htbp]\centering
\includegraphics[width=.98\textwidth]{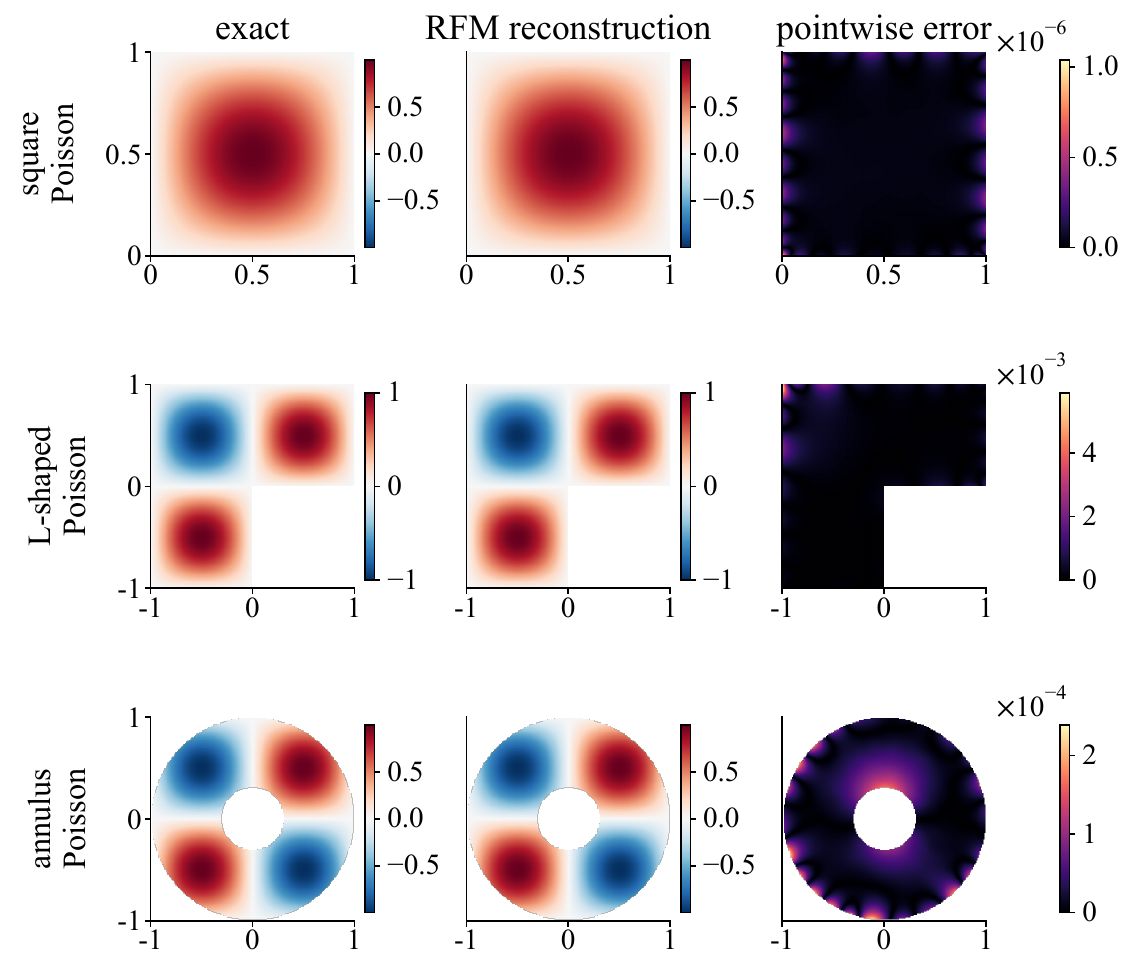}
\caption{Poisson fields for one residual-Christoffel realization. Columns show the exact field, RFM
reconstruction, and pointwise absolute error. Rows show the square ($M=400$, $N=1400$, $r=190$),
L-shape ($M=400$, $N=1400$, $r=337$), and annulus ($M=400$, $N=1400$, $r=323$). The corresponding
relative $L^2$ errors are $1.5\times10^{-7}$, $6.4\times10^{-4}$, and $7.5\times10^{-5}$.}
\label{fig:solutions}
\end{figure}

\subsection{Ablations and controls}
\noindent
This subsection presents controlled ablations that isolate five effects otherwise easily
conflated with residual-space sampling: boundary weighting, boundary-aware whitening, solver
preconditioning, column pruning after assembly, and trial-space deficiency. Except where explicitly
varied, the boundary convention is the rebalanced system with total boundary weight $\beta=1$.

The first controls concern the appended boundary block. Interior-only whitening stabilizes the
sampled residual rows; the condition number of the full matrix additionally depends on the boundary block.
Figure~\ref{fig:ablation}(a) varies only the augmented whitening Gram
$G_\alpha=(1-\alpha)G+\alpha G_b$. Incorporating the boundary Gram reduces the full-system condition
number by nearly three orders of magnitude while leaving the error unchanged. Boundary-aware
whitening therefore stabilizes the appended block without altering the collocation points or boundary
penalty.

The boundary penalty controls the relative strength of boundary enforcement.
Figure~\ref{fig:ablation}(b) and Table~\ref{tab:boundary-beta} show the same error-conditioning
tradeoff across the scalar operators: stronger boundary enforcement improves accuracy while increasing
the condition number. The greedy error varies by only about a factor of two over seven orders of
magnitude in its regularization parameter, as shown in panel (c). Panel (d) isolates whitening itself;
removing it returns the sampled system to numerical rank deficiency.

\begin{figure}[!htbp]\centering
\includegraphics[width=.98\textwidth]{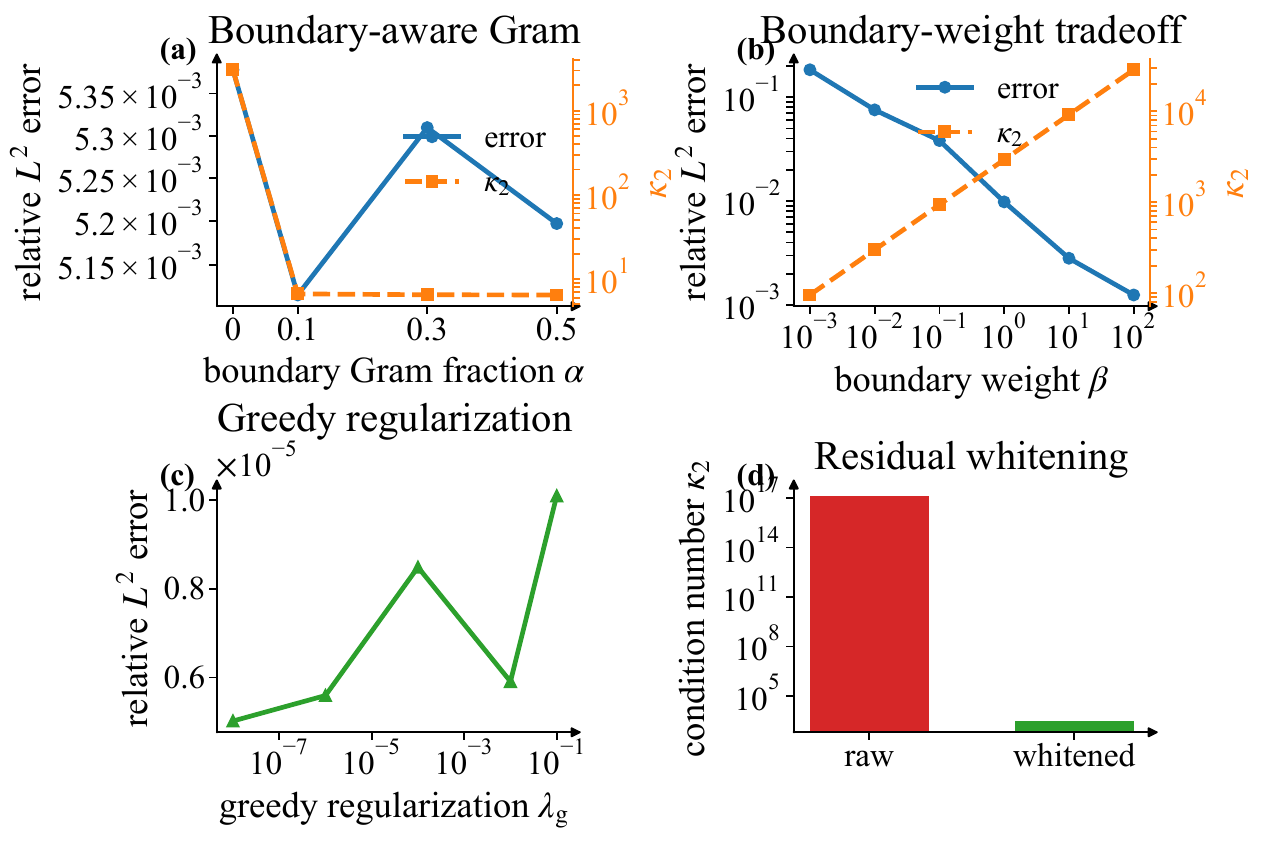}
\caption{Ablation tests on 2D Poisson. (a) Boundary-aware whitening varies $\alpha$ while holding the
collocation set, assembly weights, and $r=222$ fixed; the condition number decreases without a material error change.
(b) Increasing $\beta$ decreases the error but increases the condition number. (c) The greedy error varies by about
a factor of two over $\lambda_{\rm g}=10^{-8}$--$10^{-1}$. (d) Omitting residual whitening makes the
sampled matrix numerically rank deficient.}
\label{fig:ablation}
\end{figure}

\begin{table}[!htbp]\centering\scriptsize
\caption{Boundary-penalty sensitivity across scalar operators. Entries are median relative $L^2$
error / $\kappa_2$ over three residual-Christoffel point-selection realizations with $M=400$, $N=1500$, and
$N_b=300$.}
\label{tab:boundary-beta}
\setlength{\tabcolsep}{4pt}
\begin{tabular}{@{}lrrrr@{}}
\toprule
Operator & $r$ & $\beta=10^{-2}$ & $\beta=1$ & $\beta=10^2$ \\
\midrule
Poisson & 190 & $7.7{\times}10^{-2}$ / $4.1{\times}10^2$ &
$1.1{\times}10^{-2}$ / $3.8{\times}10^3$ &
$1.3{\times}10^{-3}$ / $3.6{\times}10^4$ \\
Helmholtz & 163 & $1.1{\times}10^{0}$ / $1.2{\times}10^1$ &
$6.5{\times}10^{-2}$ / $1.1{\times}10^2$ &
$5.2{\times}10^{-3}$ / $1.1{\times}10^3$ \\
Convection--diffusion & 177 & $2.7{\times}10^{-2}$ / $1.1{\times}10^3$ &
$1.5{\times}10^{-3}$ / $1.1{\times}10^4$ &
$3.0{\times}10^{-4}$ / $1.1{\times}10^5$ \\
\bottomrule
\end{tabular}
\end{table}

The next control fixes the collocation set and separates point selection, solver preconditioning, and
column pruning. Direct SVD supplies the accuracy reference for each row
set. Unpreconditioned LSQR, ridge-LSQR, and the least-squares solver with randomized normal projection
(LSRN) \cite{AvronToledo2010,MengSaundersMahoney2014} act on the assembled system; rank-revealing QR
factorization (RRQR), implemented by column-pivoted QR, changes the retained feature columns
\cite{RRQR2025}. Table~\ref{tab:competitive} reports the resulting accuracy, iterations, and algebraic
solve times.

Unpreconditioned and ridge-transformed LSQR reach the iteration cap, whereas LSRN converges in about
seventy iterations for both row sets. Thus solver preconditioning remains effective after either
collocation rule has been chosen. RRQR pruning to $k=r$ removes essential approximation directions;
retaining larger column sets recovers accuracy, with the nonmonotone errors confirming that column
identity matters as well as column count.

\begin{table}[!htbp]\centering\scriptsize
\caption{Fixed-collocation-set comparison on high-frequency Poisson ($M=1200$, $r=427$,
$N=1068\approx2.5r$).
Methods act through point selection, solver preconditioning, or feature pruning. Times exclude point selection, assembly,
and error evaluation; LSRN and RRQR include their listed algebraic steps, while ridge-transform
construction is excluded from ridge-LSQR. Times are single-run measurements of the listed stages.}
\label{tab:competitive}
\setlength{\tabcolsep}{2.2pt}
\begin{tabular}{@{}p{0.30\textwidth}p{0.22\textwidth}ccc@{}}
\toprule
Method & Component varied & rel.\ $L^2$ error & LSQR iters & time (s) \\
\midrule
uniform + direct SVD & SVD reference & $4.1\times10^{-5}$ & -- & 0.21 \\
uniform + LSQR & LSQR solve & $2.9\times10^{-2}$ & 3000 & 0.42 \\
uniform + LSRN & solver precond. & $2.0\times10^{-4}$ & 68 & 0.41 \\
uniform + ridge-LSQR & ridge solve & $1.0\times10^{-3}$ & 3000 & 0.42 \\
uniform + RRQR, $k=r$ & feature pruning & $1.6\times10^{-2}$ & -- & 0.60 \\
uniform + RRQR, \mbox{$k=1.5r$} & feature pruning & $8.5\times10^{-5}$ & -- & 0.84 \\
uniform + RRQR, $k=2r$ & feature pruning & $1.2\times10^{-4}$ & -- & 0.80 \\
Christoffel + direct SVD & point selection & $1.6\times10^{-5}$ & -- & 0.25 \\
Christoffel + ridge-LSQR & selection + ridge map & $1.5\times10^{-3}$ & 3000 & 0.41 \\
Christoffel + LSRN & selection + LSRN & $1.8\times10^{-4}$ & 70 & 0.43 \\
\bottomrule
\end{tabular}
\end{table}

The RRQR rows show one side of the approximation-space issue: removing too many feature columns can
destroy accuracy even when the collocation set is fixed. A complementary control examines whether
point selection can
compensate for a missing solution direction. On the singular L-shaped problem, a smooth RFM with
$8000$ interior points leaves a relative $L^2$ error of $1.4\times10^{-2}$, whereas appending the known
corner-singularity function reaches the numerical precision of the solve. This enrichment separates
trial-space deficiency from point placement: residual leverage stabilizes represented directions,
whereas enrichment expands the approximation space.

\subsection{Computational regime and approximate leverage construction}
Figure~\ref{fig:scaling-solvers} compares dense post-assembly solvers on the high-frequency Poisson
benchmark. Direct SVD remains the fastest and most accurate method over the displayed feature range.
Among the iterative methods, LSRN is fastest and most accurate. Unpreconditioned LSQR and ridge-LSQR
are slower and less accurate. The residual-Christoffel curve reports ridge-LSQR on the selected rows.

\begin{figure}[!htbp]\centering
\includegraphics[width=.98\textwidth]{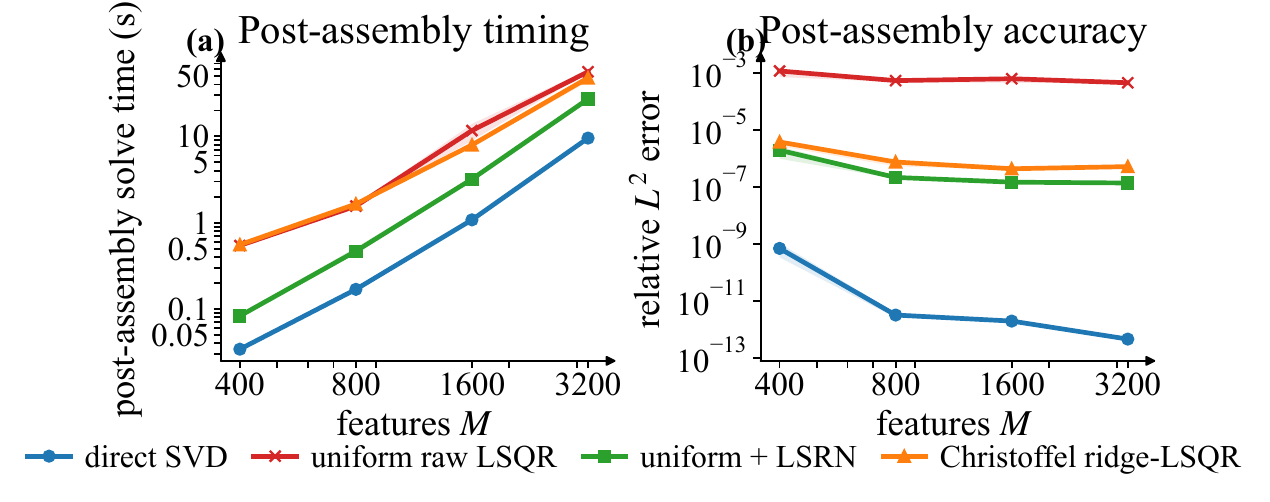}
\caption{Post-assembly methods on high-frequency Poisson with $N=2M$. Direct SVD, unpreconditioned LSQR, and LSRN
use uniform rows; ridge-LSQR uses residual-Christoffel rows. Panels show solve time and relative $L^2$
error, summarized by medians and interquartile bands over five independent feature, quadrature, and
sampling realizations. Timings exclude leverage construction, point selection, boundary construction,
matrix assembly, and error evaluation. Direct
SVD remains fastest and most accurate throughout the displayed dense regime.}
\label{fig:scaling-solvers}
\end{figure}

Figure~\ref{fig:sketched-leverage} compares dense Gram eigendecomposition with the randomized
range-finder construction from Section~\ref{sec:design-cost}. Beyond the observed crossover, the sketch
reduces both construction time and dominant-array storage, while its smaller retained residual subspace
produces higher PDE error. The crossover therefore exposes a direct cost-accuracy tradeoff in leverage
construction.

\begin{figure}[H]\centering
\includegraphics[width=.80\textwidth]{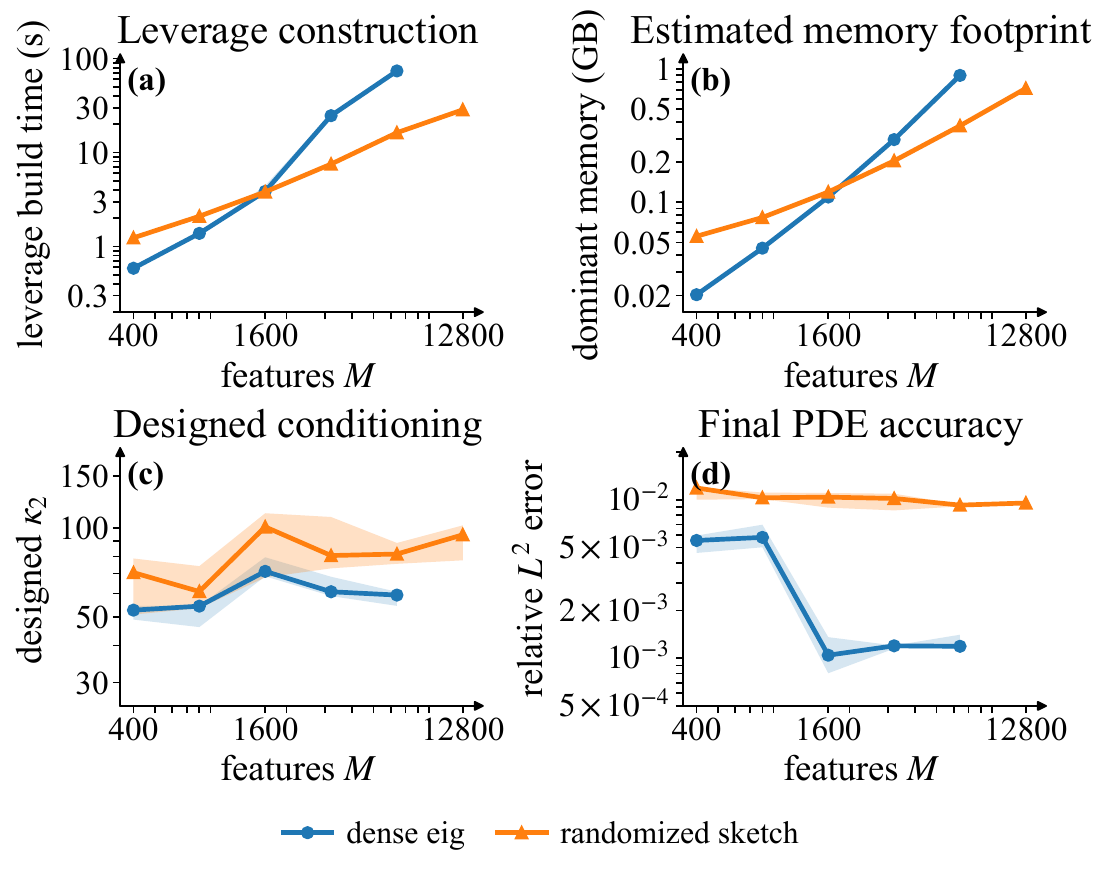}
\caption{Sketched residual-leverage construction. The randomized range finder is applied to
$R_Q=W^{1/2}\Psi_Q$ and compared with dense Gram eigendecomposition. Curves show medians over five
realizations; panel (b) reports the dominant-array storage estimate. For $M\ge3200$, where both
constructions were run, the sketch lowers construction time and dominant-array storage but retains a
smaller residual subspace and gives higher PDE error.}
\label{fig:sketched-leverage}
\end{figure}

\FloatBarrier
\section{Conclusion and outlook}\label{sec:discussion}

This paper developed an operator-aware collocation strategy for linear PDEs in fixed random feature
spaces. The operator-residual Gram defines the residual-Christoffel density, inverse-density weights,
and coefficient whitening map; a deterministic scalar-row companion maximizes regularized
log-determinant increments. Conditional on the trial space, the randomized construction approximates
the whitened residual Gram with $O(r\log(r/\delta)/\varepsilon^2)$ samples. For uniformly analytic
residual kernels, spectral decay yields a polylogarithmic ridge effective dimension. Experiments across
scalar and vector PDEs, varied geometries, and one to three dimensions show substantially improved
conditioning and iterative convergence, together with the lowest small-sample scalar errors from the
greedy construction. These results establish residual-space geometry as the organizing principle for
stable strong-form random feature collocation.

Natural extensions combine localized or enriched trial spaces with operator-aware sampling for
nonsmooth and multiscale solutions, use streamed residual-Gram factorizations for larger feature sets,
and employ block leverage with augmented interior-boundary Grams for coupled or parameterized PDEs.
Each direction preserves the principle of resolving the operator-induced geometry of the trial space.

\appendix

\section{Proofs of the spectral-decay results}\label{app:kernel}
This appendix collects the proofs of the kernel-spectral results of Section~\ref{sec:effdim}.

\begin{proof}[Proof of Theorem~\ref{thm:oplim}]
$S_M=\Phi_M\Phi_M^*$ is cospectral with $G=\Phi_M^*\Phi_M$, and
$M^{-1}S_M$ averages i.i.d.\ PSD operators
$X_m=\zeta_{\theta_m}\otimes\zeta_{\theta_m}$ with $\E X_m=\mathcal T$ and
$\|X_m\|_{\mathrm{HS}}=\|\zeta_{\theta_m}\|_H^2\le R_\zeta$, where
$\|\cdot\|_{\mathrm{HS}}$ is the Hilbert--Schmidt norm. Independence and centering give
\begin{align*}
  \E\left\|M^{-1}S_M-\mathcal T\right\|_{\mathrm{HS}}^2
  &=\frac1M\E\|X_1-\mathcal T\|_{\mathrm{HS}}^2\\
  &=\frac1M\left(\E\|X_1\|_{\mathrm{HS}}^2
    -\|\mathcal T\|_{\mathrm{HS}}^2\right)\\
  &\le\frac1M\E\|\zeta_{\theta_1}\|_H^4
  \le\frac{R_\zeta\,\mathrm{tr}(\mathcal T)}{M}.
\end{align*}
Since the operator norm is bounded by the Hilbert--Schmidt norm, Markov's inequality yields
\[
  \Pr\!\left[\|M^{-1}S_M-\mathcal T\|_2>\varepsilon_{\rm op}\tau_1\right]
  \le \frac{R_\zeta\,\mathrm{tr}(\mathcal T)}{M\varepsilon_{\rm op}^2\tau_1^2}.
\]
Condition \eqref{eq:operator-limit-sample-condition} makes this probability at most $\delta$.
The compact self-adjoint eigenvalue perturbation inequality transfers the operator-norm bound to the
eigenvalues, with the finite-rank spectrum of $M^{-1}S_M$ padded by zeros. Finally, the
Hilbert--Schmidt operator space over the separable space $H=L^2(\mu)$ is itself separable, and
$\E\|X_1\|_{\mathrm{HS}}<\infty$. The Hilbert-space strong law therefore gives
$\|M^{-1}S_M-\mathcal T\|_{\mathrm{HS}}\to0$ almost surely, which proves the stated limiting claim.
\end{proof}

\begin{proof}[Proof of Theorem~\ref{thm:decay}]
Analyticity is needed only in the first kernel variable. The uniform continuation gives a
tensor-product Chebyshev expansion of
$k_L(\cdot,\bx')$ on $D$, with coefficients bounded uniformly in $\bx'$. Truncating the expansion at
degree $n$ in each coordinate produces a kernel, where $q_\alpha$ is a tensor-product Chebyshev
polynomial and $c_\alpha$ is its coefficient function in the second kernel variable,
\[
  k_{L,n}(\bx,\bx')
  =\sum_{\alpha\in\{0,\ldots,n\}^d}q_\alpha(\bx)c_\alpha(\bx')
\]
such that
\[
  \sup_{\bx,\bx'\in\Omega}|k_L(\bx,\bx')-k_{L,n}(\bx,\bx')|
  \le C_0e^{-c_0n},
\]
where $C_0,c_0>0$ depend only on the analytic extension in the theorem. The integral operator
$\mathcal T_n$ associated with $k_{L,n}$ has range in the span of the
$(n+1)^d$ tensor-product polynomials $q_\alpha$, so
$\mathrm{rank}(\mathcal T_n)\le(n+1)^d$. Because $\mu$ is a probability measure,
\[
  \|\mathcal T-\mathcal T_n\|_2
  \le\|k_L-k_{L,n}\|_{L^2(\mu\times\mu)}
  \le C_0e^{-c_0n}.
\]
The approximation-number characterization of the eigenvalues of a compact PSD operator therefore
gives $\tau_{(n+1)^d+1}\le C_0e^{-c_0n}$. Choosing $n$ comparable to $j^{1/d}$ and enlarging the
constant to cover the finitely many leading indices yields
$\tau_j\le C_1e^{-c_1j^{1/d}}$.
\end{proof}

\begin{proof}[Proof of Corollary~\ref{cor:deff}]
Set
\[
  j^\star=\max\!\left\{1,
  \left\lceil\left(c_1^{-1}\log(C_1/\ell)\right)^d\right\rceil\right\}.
\]
Here $\lceil\cdot\rceil$ denotes the ceiling function.
The first $j^\star$ terms contribute at most $j^\star$. For $j>j^\star$,
$\tau_j/(\tau_j+\ell)\le\tau_j/\ell$, and monotonicity gives
\[
  \sum_{j>j^\star}\frac{\tau_j}{\ell}
  \le \frac{C_1}{\ell}\int_{j^\star}^{\infty}e^{-c_1t^{1/d}}\,dt
  =\frac{C_1d}{\ell}\int_{(j^\star)^{1/d}}^\infty
    u^{d-1}e^{-c_1u}\,du.
\]
Repeated integration by parts, or the elementary exponential-tail bound for a fixed polynomial,
shows that the last integral is at most
$C e^{-c_1a}(1+a^{d-1})$ with $a=(j^\star)^{1/d}$ and a constant $C>0$ depending only on $c_1$ and
$d$. By the definition of $j^\star$,
$(C_1/\ell)e^{-c_1a}\le1$. Hence the tail is
$O((1+\log(C_1/\ell))^{d-1})$, while the head is
$O((1+\log(C_1/\ell))^d)$, which proves the claim.
\end{proof}

\section*{Acknowledgements}
This work was carried out during Jiale Linghu's visit to Professor Weizhu Bao at the National
University of Singapore. Both authors thank Professor Bao for his valuable advice and helpful
discussions.

\section*{Data availability}
An initial code release is publicly available in the RC-RFM repository at
\url{https://github.com/linghujiale/RC-RFM}. The complete experiment scripts and numerical records
supporting this manuscript will be deposited in the same repository before publication and are
available from the corresponding author on reasonable request.

\section*{Funding}
The authors declare that no funds, grants, or other support were received during the preparation of
this manuscript.

\section*{Competing interests}
The authors have no relevant financial or non-financial interests to disclose.

\section*{Ethics approval}
Not applicable. This article does not contain studies with human participants or animals performed by
any of the authors.

\section*{Author contributions}
Jiale Linghu and Yangshuai Wang contributed equally to the conceptualization, methodology, formal
analysis, investigation, validation, numerical implementation, interpretation of results, and manuscript
writing. Both authors reviewed, edited, and approved the final manuscript.


\end{document}